\documentclass[11pt,a4paper,reqno,notitlepage]{article}

\title{Dimension theory and degenerations of \\ de Jonqui\`eres divisors}
\author{Mara Ungureanu}
\date{}

\AtEndDocument{\medskip
\noindent\begin{footnotesize}\textsc{Albert-Ludwigs-Universit\"at Freiburg, Mathematisches Institut, Abteilung Reine Mathematik, Ernst-Zermelo-Str 1, 79104 Freiburg}
\smallskip\\
\textit{E-mail address:} \texttt{mara.ungureanu@math.uni-freiburg.de}\end{footnotesize}
}

\usepackage[style=alphabetic,maxnames=99,maxalphanames=6]{biblatex}
\usepackage{amsmath, amssymb, amsthm,mathrsfs}
\usepackage{enumitem}
\usepackage{color}
\usepackage{mathtools}
\usepackage[a4paper,left=2.5cm,right=2.5cm,bottom=3cm]{geometry}
\usepackage{indentfirst}
\usepackage{tikz}
\usetikzlibrary{matrix,arrows}
\usepackage{epstopdf}
\usepackage{float}
\usepackage{lmodern}
\usepackage{microtype}

\usepackage[utf8]{inputenc}
\usepackage[T1]{fontenc}

\newcommand{\grd}{g^r_d}

\newcommand{\grdop}[2]{g^{#1}_{#2}}
\DeclareMathOperator{\im}{\text{im}}
\DeclareMathOperator{\coker}{\text{coker}}
\DeclareMathOperator{\rk}{rk}
\DeclareMathOperator{\oo}{\mathcal{O}}

\DeclareMathOperator{\p}{\mathbb{P}}
\DeclareMathOperator{\pic}{\text{Pic}}
\DeclareMathOperator{\ord}{\text{ord}}
\newcommand{\Grd}{G^r_d(C)}
\newcommand{\Wrd}{W^r_d(C)}
\newcommand{\Crd}{C^r_d}
\DeclareMathOperator{\hilb}{\text{Hilb}_{d,g,r}}

\newtheorem{thm}{Theorem}[section] 
\newtheorem{lemma}[thm]{Lemma}
\newtheorem{prop}[thm]{Proposition}
\newtheorem{cor}[thm]{Corollary}

\theoremstyle{definition}
\newtheorem{defi}{Definition}[section]
\theoremstyle{definition}

\theoremstyle{remark}
\newtheorem{rem}{Remark}[section]

\begin{document}

\maketitle

\begin{abstract}
This paper aims at settling the issue of the validity of the de Jonquières formulas.  
We consider the space of divisors with prescribed multiplicity, or de Jonqui\`eres divisors, contained in a linear series on a smooth projective curve.  
Assuming zero expected dimension of this space,  the de Jonqui\`eres formulas compute the virtual number of de Jonqui\`eres divisors.  Using degenerations to nodal curves we show that for a general curve equipped with a general complete linear series, the space is of expected dimension, which shows that the counts are in fact true.  This implies that in the case of negative expected dimension a general linear series on a general curve does not admit de Jonqui\`eres divisors of the expected type. 
\end{abstract}

\section{Introduction}

In his 1866 memoir \autocite{deJo}, de Jonqui\`eres computed the number of divisors with prescribed multiplicities that are contained in a fixed linear series on a given plane algebraic curve.  

Almost a century later, and using modern techniques of topology and intersection theory of cycles on the symmetric product on a curve, MacDonald \autocite{Mac} and Mattuck \autocite{Mat}  recovered the original formula in characteristic zero and arbitrary characteristic, respectively, and generalised it to linear series of any dimension.  
However, their work does not address the vagueness of the classical statements, assuming either that the linear series in question is sufficiently generic, or that the multiplicities are counted correctly.
To address this issue, Vainsencher \autocite{Va} described the locus of divisors with prescribed multiplicities as the vanishing locus of a section of a bundle of the appropriate rank.  Using a natural filtration of this bundle, he computed its Chern classes without making use of the Grothendieck-Riemann-Roch theorem, and established the enumerative validity of the de Jonqui\`eres formula for plane curves and for some higher dimensional cases.

The aim of this paper is twofold. On the one hand, we settle the issue of the validity of the de Jonqui\`eres formula for linear series of arbitrary degree and dimension on a general curve by studying the geometry of the respective moduli space.  
On the other hand, we develop a theory of degenerations for de Jonquières divisors to nodal curves, which plays a central role as the main tool of in the study of the aforementioned moduli space.

Aside from being interesting objects in their own right, de Jonquières divisors and their degenerations are natural generalisations of the concept of strata of abelian differentials which were first introduced in the context of Teichm\"uller dynamics and flat surfaces - see the works of Masur \autocite{Mas} and Veech \autocite{Ve}, and more recently, of Bainbridge, Chen, Gendron, Grushevsky, and M\"oller \autocite{BCGGM}.  These strata are, however,  interesting objects also from the point of view of algebraic geometry, as can be seen in the work of Farkas and Pandharipande \autocite{FP}, Chen and Tarasca \autocite{CT}, or Mullane \autocite{Mu}.  In fact, the result of Polishchuk \autocite{Po} concerning the dimension of the strata in $\mathcal{M}_{g,n}$ provides an important clue towards the validity of the de Jonquières formulas.

In what follows we fix the notation and describe the objects of interest.
Let $C$ be a smooth projective curve of genus $g$ and denote by $C_d$ its $d$-th symmetric product.  Furthermore, let $\Grd$ parametrise linear series of type $\grd$, i.e.
\[\Grd:=\{ l=(L,V) \mid L\in\pic^d(C), V\in G(r+1,H^0(C,L)) \}.\]  
In this paper we focus on the case of Brill-Noether general curves, meaning that $\Grd$ is a smooth variety and its dimension is given by the Brill-Noether number
\begin{equation}\label{eq:brillnoether} 
\rho(g,r,d)=g-(r+1)(g-d+r)\geq 0. 
\end{equation}
Moreover it follows that the Hilbert scheme $\hilb$ parametrising curves in $\p^r$ of degree $d$ and (arithmetic) genus $g$ has a unique component $\mathcal{H}_{d,g,r}$, whose general point corresponds to a smooth curve and which maps dominantly onto $\mathcal{M}_g$.  Thus, when we talk about a general curve $C$ with a general linear series $l\in\Grd$ we refer to a general point 
$[C \xrightarrow{l} \p^r]\in\mathcal{H}_{d,g,r}$.

We now define the main object of interest in this paper:
for a fixed smooth curve $C$ of genus $g$ with a fixed linear series $l=(L,V)\in\Grd$,
a \textit{de Jonqui\`eres divisor of length} $N$ is a divisor $a_1 D_1 + \ldots + a_k D_k \in C_d$ such that
\[ a_1 D_1 + \ldots + a_k D_k \in\p V, \]
where $k\leq d$ is a positive integer and the $D_i$ are effective divisors of degree $d_i$ for $i=1,\ldots,k$ such that $N=\sum_{i=1}^k d_i$.  
If $l$ is complete (i.e.~such that $g-d+r\geq 0$), the definition of a de Jonqui\`eres divisor is equivalent to 
\[L\simeq\oo_C(a_1 D_1 + \ldots + a_k D_k).\]
Furthermore, if we let $\mu_1=(a_1,\ldots,a_k)$ and $\mu_2=(d_1,\ldots,d_k)$ be
two positive partitions such that $\sum_{i=1}^k a_i d_i = d$, then we denote the set of
de Jonqui\`eres divisors of length $N$ determined by $\mu_1$ and $\mu_2$ by $DJ_{k,N}^{r,d}(\mu_1,\mu_2,C,l)$.

In the particular case when $d_i=1$ for all $i=1,\ldots,k$, let $n:=N=k$ and the de Jonqui\`eres divisor is of the form
\[ a_1 p_1  + \ldots + a_n p_n,\]
for some distinct points $p_1,\ldots,p_n \in C$.
Here we simplify the notation to
\[DJ_{k,N}^{r,d}(\mu_1,\mu_2,C,l)=DJ^{r,d}_n(\mu_1,C,l).\]

It turns out (see Section \ref{sec:deglocus}) that the space $DJ_{k,N}^{r,d}(\mu_1,\mu_2,C,l)$ 
has the structure of a determinantal variety and its expected dimension (or, equivalently, lower bound for its dimension) is
\[ \text{exp}\dim DJ_{k,N}^{r,d}(\mu_1,\mu_2,C,l) = N-d+r.\]
The de Jonqui\`eres formula (cf.~\autocite{Mat} \S 5) states that, if we expect there to be a finite number of de Jonqui\`eres divisors of length $N$ (so if $N-d+r=0$), then this virtual number is given by the coefficient of the monomial $t_1^{d_1}\cdot\ldots\cdot t_k^{d_k}$ in
\begin{equation}\label{eq:dejformula}
(1+a_1^2 t_1 + \ldots + a_k^2 t_k)^{g}(1+a_1 t_1 + \ldots + a_k t_k)^{d-r-g}.
\end{equation} 
Substituting $r=1$ and $d_1=\ldots=d_k=1$ in formula (\ref{eq:dejformula}) recovers the number of ramification points of a Hurwitz cover of $C$ obtained from the Pl\"ucker formula.  In addition, if $C$ is the plane cubic, then $g=1$, $d-r-g=1$ and we recover its 9 flex points.  Lastly, taking the linear series to be the canonical one, we recover the number of odd theta characteristics on a smooth general curve.  Hence we expect these counts to be true.  To settle the issue we must however study the space $DJ_{k,N}^{r,d}(\mu_1,\mu_2,C,l)$, establish whether it is empty when the expected dimension is negative, and when non-empty whether it is smooth, reduced, and of expected dimension.

Luckily, we are able to settle these questions in the affirmative.  In fact, the non-emptiness of the space of de Jonqui\`eres when the expected dimension is non-negative follows from an easy class computation, which we explain in Section \ref{sec:existencedej}.

The questions regarding the dimension of the space $DJ_{k,N}^{r,d}(\mu_1,\mu_2,C,l)$ and whether it is empty when the expected dimension is negative are less straightforward and require the degeneration techniques.  Using limit linear series on nodal curves of compact type, we prove
\begin{thm}[Dimension theorem]\label{dimension}
Fix a general curve $C$ of genus $g$ equipped with a general complete linear series $l\in\Grd$. If $N-d+r\geq 0$, the space $DJ_{k,N}^{r,d}(\mu_1,\mu_2,C,l)$ is of expected dimension,
\[\dim DJ_{k,N}^{r,d}(\mu_1,\mu_2,C,l) = N-d+r.\]
\end{thm}
\noindent A direct consequence of the dimension theorem is the non-existence statement for complete linear series:
\begin{cor}
Let $C$ be a general curve equipped with a general complete linear series $l\in\Grd$.  If $N-d+r<0$, the variety $DJ_{k,N}^{r,d}(\mu_1,\mu_2,C,l)$ is empty.
\end{cor}

The validity of de Jonquières' counts is a direct consequence of Theorem \ref{dimension}, and of the determinantal variety structure of the space of de Jonquières divisors.  The latter implies that $DJ_{k,N}^{r,d}(\mu_1,\mu_2,C,l)$ is in fact a Cohen-Macaulay variety (Proposition 4.1, Chapter II,  \autocite{ACGH}).  As such, if it is zero-dimensional, it consists of a finite number of discrete closed points. This yields

\begin{cor}\label{cor:counts}
Let $C$ be a general curve equipped with a general complete linear series $l\in\Grd$.  If $N-d+r=0$, the variety $DJ_{k,N}^{r,d}(\mu_1,\mu_2,C,l)$ is a finite collection of reduced points.
\end{cor}

We address the issue of the smoothness of $DJ_{k,N}^{r,d}(\mu_1,\mu_2,C,l)$ by expressing the space as an intersection of subvarieties inside the symmetric product $C_d$ and obtaining a transversality condition from the study of the relevant tangent spaces.
\begin{thm}[Smoothness result]\label{thm:smooth}
Let $C$ be a smooth general curve of genus $g$.  Then for any complete linear series $l\in G^r_d(C)$, the space $DJ_{k,N}^{r,d}(\mu_1,\mu_2,C,l)$ is smooth, whenever $N-d+r>0$.
\end{thm}
\noindent The proof is also by degeneration to nodal curves and limit linear series, however this time using a strategy developed in \autocite{Fa2} .

Finally, we prove the non-existence result for non-complete linear series using a different degeneration technique, namely compactified Picard schemes for moduli of stable pointed curves.  We obtain
\begin{thm}[Non-existence statement]\label{cor:nonex}
Let $C$ be a general curve equipped with a general linear series $l\in\Grd$ satisfying $g-d+r<0$ and let $\mu$ be a positive partition of $d$ length $n$.  If $n-d+r<0$, the variety $DJ^{r,d}_n(\mu,C,l)$ is empty.
\end{thm}
\noindent These degenerations are only suitable for treating the case of de Jonqui\`eres divisors satisfying $d_1=\ldots=d_k=1$, as they rely on the fact that the points in the support of the divisor are distinct.

The paper is organised as follows.  In Section \ref{sec:geomdej} we collect some preliminary results about the space $DJ_{k,N}^{r,d}(\mu_1,\mu_2,C,l)$ which will form the basis for the arguments in the remainder of the paper.
In Section \ref{sec:deglocus} the space $DJ_{k,N}^{r,d}(\mu_1,\mu_2,C,l)$ is endowed with the structure of a determinantal subvariety of $C_d$, while in Section \ref{sec:existencedej}
an easy argument shows that the space $DJ_{k,N}^{r,d}(\mu_1,\mu_2,C,l)$ is indeed non-empty when the expected dimension is positive.
Furthemore, in Section \ref{sec:nonspecial} we express the space $DJ_{k,N}^{r,d}(\mu_1,\mu_2,C,l)$ as an intersection of subvarieties of the symmetric product $C_d$ and establish
the condition for this intersection to be transverse.

We then proceed in Section \ref{nodes} to describe degenerations of de Jonqui\`eres divisors for families of smooth curves with a nodal central fibre.  We approach the question from two perspectives:
\begin{enumerate}[wide, labelwidth=!, labelindent=0pt]
	\item In Section \ref{sec:lls} we look at the theory of limit linear series, as developed by Eisenbud and Harris in \autocite{EH86}, which applies to the case when the central fibre is a curve of compact type.  In particular, we address the question of what it means to say that a limit linear series admits a de Jonquières divisor.  In this context we use a very simple degeneration and an induction argument to prove Theorem \ref{dimension} in Section \ref{sec:prooflls} and Theorem \ref{thm:smooth} in Section \ref{sec:smooth}.
	\item In Section \ref{sec:compactified} we discuss compactifications of the Picard schemes over the moduli space of stable curves with marked points, following \autocite{Ca} and \autocite{Me}, which come into play for central fibres that are stable curves.  This approach represents our attempt at generalising the degenerations of abelian differentials, as they appear in \autocite{Ch} or \autocite{FP}.  Of course, what distinguishes our case from the case of differentials is that we do not have a readily available equivalent of the relative dualising sheaf on the family of curves.  
These degenerations have as direct practical application the proof by induction of Theorem \ref{cor:nonex} in Section \ref{sec:noncomplete}.
	\end{enumerate}  

We conclude the paper in Section \ref{sec:lorentz} with a discussion of the space de Jonquières divisors where the partition $\mu_1$ is allowed to have negative coefficients as well.
 


\section*{Acknowledgements}
This paper is part of my PhD thesis.  I would like to thank my advisor Gavril Farkas for introducing me to this circle of ideas and for his suggestions to approaching this problem.  I am also grateful to Irfan Kad\i k\"oyl\"u and Nicola Pagani for useful discussions and to the referee for a careful reading of the paper and their recommendations for improvements.

\section{The geometry of the space of de Jonquières divisors}\label{sec:geomdej}
	In this section we extract as much information as possible about the geometry of the space $DJ_{k,N}^{r,d}(\mu_1,\mu_2,C,l)$ of de Jonquières divisors for a fixed curve $C$ equipped with a linear series $l=(L,V)$ of type $\grd$, without using any degeneration techniques. 
	
	We recall some standard definitions: for a smooth curve $C$, let $\Crd$ be the subvariety of $C_d$ parametrising effective divisors of degree $d$ on $C$ moving in a linear series of dimension at least $r$:  
\[ \Crd := \{ D\in C_d \mid \dim | D | \geq r \}, \]
and $\Wrd$ be the associated variety of complete linear series of degree $d$ and dimension at least $r$, i.e.
\[ \Wrd := \{ L\in\pic^d(C) \mid h^0(C,L)\geq r+1 \}\subseteq\pic^d(C).\]
Allowing the curve to vary in the moduli space $\mathcal{M}_g$ of curves of genus $g$, we denote by $\mathcal{W}^r_d$ the relative counterpart of $\Wrd$.
	\subsection{The space of de Jonquières divisors as degeneracy locus}\label{sec:deglocus}
\indent Fix an integer $k\leq d$ and two vectors of positive integers $\mu_1=(a_1,\ldots,a_k)$ and $\mu_2=(d_1,\ldots,d_k)$ such that $\sum_{i=1}^{k} a_i d_i =  d$.  The space $DJ_{k,N}^{r,d}(\mu_1,\mu_2,C,l)$ can be described as a degeneracy locus of vector bundles over $C_d$ as follows: the idea is that the condition 
\[ a_1 D_1 + \ldots + a_k D_k \in\p V \]
is equivalent to the condition that the natural restriction map
\[ V\rightarrow H^0(C,L/L(-a_1 D_1-\ldots-a_k D_k)) \]
has non-zero kernel.  
To reformulate this globally in terms of a morphism of two vector bundles over $C_d$, let the first bundle $\mathcal{E}=\oo_{C_d}\otimes V$ be the trivial bundle.  As for the second bundle, consider the diagram 
 
 \begin{figure}[H]\centering
  \begin{tikzpicture}
    \matrix (m) [matrix of math nodes,row sep=2em,column sep=1em,minimum width=1em]
  {
     & C \times C_d & \supset \mathcal{U} \\
     C & & C_d\\};
  \path[-stealth]
    (m-1-2) edge node [auto,swap] {$\sigma$} (m-2-1)
            edge node [auto]{$\tau$}  (m-2-3);
 \end{tikzpicture}
 \end{figure}

  
\noindent where $\sigma$ and $\tau$ are the usual projections and $\mathcal{U}$ is the universal divisor defined as
\[\mathcal{U} = \{ (p,D) \mid D\in C_d \text{ and }p\in D \} \subset C\times C_d.\]
Alternatively, identifying $C_d$ with the Hilbert scheme $C^{[d]}$ of $d$ points on $C$, one defines $\mathcal{U}$ as the universal family $\mathcal{U}\subset C\times C^{[d]}$.
For the second bundle, consider the sheaf:
\[ \mathcal{F}_d(L) = \tau_*(\sigma^* L\otimes \oo_{\mathcal{U}}), \]
By cohomology and base change $\mathcal{F}_d(L)$ is indeed a vector bundle. 
The fibre of $\mathcal{F}_d(L)$ over any point $D\in C_d$ is given by the $d$-dimensional vector space 
$H^0(C,L/L(-D))$.

Finally, let $\Phi:\mathcal{E}\rightarrow \mathcal{F}_d(L)$ be the vector bundle morphism obtained by pushing down to $C_d$ the restriction $\sigma^*L\rightarrow\sigma^*L\otimes\oo_{\mathcal{U}}$. Moreover let
 \[\Sigma_{k,N}(\mu_1,\mu_2) = \Bigl\{ E \in C_d \mid E = \sum_{i=1}^k a_i D_i\text{ for some }D_1\in C_{d_1},\ldots,D_k\in C_{d_k}\}.\]

\noindent The space $DJ_{k,N}^{r,d}(\mu_1,\mu_2,C,l)$ is defined as the $r$\textit{-th degeneracy locus} of $\Phi$, i.e.~the locus in $\Sigma_{k,N}(\mu_1,\mu_2)$ where $\rk\Phi\leq r$.

\begin{lemma}\label{lemma:dimbound}
For every point $D\in DJ_{k,N}^{r,d}(\mu_1,\mu_2,C,l)$, one has
\[ \dim_D DJ_{k,N}^{r,d}(\mu_1,\mu_2,C,l) \geq N-d+r.\] 
\end{lemma}
\begin{proof}
From the description of $DJ_{k,N}^{r,d}(\mu_1,\mu_2,C,l)$ as a degeneracy locus, its codimension in $\Sigma_{k,N}(\mu)$ is at most
\[ (\rk \mathcal{E}-r)(\rk \mathcal{F}_d(L) -r)=(r+1-r)(d-r)=d-r.  \]
Since $\dim\Sigma_{k,N}(\mu_1,\mu_2)=N$, the dimension estimate follows.
\end{proof}

Finally, we record here an easy result that forms the base case for the induction argument in the proof of Theorem \ref{cor:nonex}.

\begin{lemma}\label{lemma:nonspecial}
Let $C$ be any smooth curve of genus $g$ with a general linear series $l\in\Grd$.  Fix an integer $k\leq d$ and two vectors of positive integers $\mu_1=(a_1,\ldots,a_k)$ and $\mu_2=(d_1,\ldots,d_k)$ such that $\sum_{i=1}^{k} a_i d_i =  d$ and $N-d+r<0$.
\begin{enumerate}
	\item If $g-d+r=0$, then $DJ_{k,N}^{r,d}(\mu_1,\mu_2,C,l)=\emptyset$.
	\item If $g-d+r<0$ and $N<g$, then $DJ_{k,N}^{r,d}(\mu_1,\mu_2,C,l)=\emptyset$.
\end{enumerate}
\end{lemma}  
\begin{proof}
(1): Consider the following restriction of the Abel-Jacobi-type map:
\begin{align*}
    u:\Sigma_{k,N}(\mu_1,\mu_2) &\rightarrow \pic^d(C) \\
    a_1 D_1+\ldots+a_k D_k &\mapsto \oo_C(a_1 D_1 + \ldots + a_k D_k).
  \end{align*}
In the non-special regime $\pic^d(C)=W^r_d(C)$.
  Moreover the image of $\varphi$ is closed and $\dim\im u\leq N<g=\dim\pic^d(C)$.   Thus a general line bundle $L\in\pic^d(C)$ is not contained in the image of $u$ whence we conclude that the divisor $a_1 D_1 + \ldots + a_k D_k$ is not contained in a general linear series $l$ of type $\grd$, i.e.~$DJ_{k,N}^{r,d}(\mu_1,\mu_2,C,l)=\emptyset$.  This is Corollary \ref{cor:nonex} for non-special linear series.  \\
As a consequence, for all $r'<r$, a general linear series $l'$ in $\text{G}^{r'}_d(C)$ also
has \[DJ_{k,N}^{r',d}(\mu_1,\mu_2,C,l')=\emptyset.\]  To see this, let $c:\text{G}^{r'}_d(C)\rightarrow\text{W}^{r'}_d\subset\Wrd$ be the forgetful map $(L,V)\mapsto L$.  Note that a line bundle $L\in\Wrd\setminus\im u$ does not admit de Jonquières divisors of length $N$.  Now, since $c$ is continuous (as the projection morphism from a Grassmann bundle), $c^{-1}(\text{W}^{r'}_d(C)\setminus\im u)$ is also open in $\text{G}^{r'}_d$ and nonempty.  Hence no $l'\in c^{-1}(\text{W}^{r'}_d(C)\setminus\im u)$ admits a de Jonquières divisor of length $N$ and our claim is proved.\\ 
(2): Set $r_1=d-g$ so that $g-d+r_1=0$ and $r<r_1$.  We conclude from the discussion above that  
if $N<g$, then $DJ_{k,N}^{r,d}(\mu_1,\mu_2,C,l)=\emptyset$ for a general linear series $l$.  The non-existence for $n\geq g$ for $DJ_n^{r,d}(\mu,C,l)$ follows by an induction argument explained in Section \ref{sec:compactified}. 
\end{proof}

\subsection{Existence of de Jonqui\`eres divisors}\label{sec:existencedej}
\indent Luckily, the question of existence is easily answered in a manner similar to that of the first proofs of the  existence part of the Brill-Noether theorem (\cite{Ke} and \cite{KL}).  The idea is to simply look at the class of $DJ_{k,N}^{r,d}(\mu_1,\mu_2,C,l)$ and establish its positiveness.
Consider the \textit{diagonal mapping} for $C_d$:
\begin{align*}
\epsilon: C_{d_1}\times\ldots\times C_{d_k} &\rightarrow C_d \\
D_1+\ldots+D_k &\mapsto a_1 D_1+\ldots +a_k D_k.
\end{align*}
It is well-known (see for example chapter VIII \S 5 of \autocite{ACGH}) that the 
image, via $\epsilon$ of the fundamental class of $\epsilon(C_{d_1}\times\ldots\times C_{d_k})$
is equal to the coefficient of the monomial $t_1^{d_1} \cdot\ldots \cdot t_k^{d_k}$ in 
\[ \sum_{a\geq b} \frac{(-1)^{a+b}}{b!(a-b)!}\left(1+\sum_{i=1}^{k}a_it_i\right)^{N-g+b} \left( 1+\sum_{i=1}^{k}a_i^2t_i \right)^{g-b}x^{d-N-a}\theta^a,\]
where $\theta$ is the pullback of the fundamental class of the theta divisor to $C_d$ and $x$ the class of the divisor $q+C_{d-1}\subset C_d$. 
Evaluating this formula on a linear series $l$ of degree $d$ and dimension $r$, and using that $\theta\mid l=0$, we obtain the following expression for the class of $DJ^{r,d}_n(\mu,C,l)$:
\[ \left(1+\sum_{i=1}^{k}a_it_i\right)^{N-g} \left( 1+\sum_{i=1}^{k}a_i^2t_i \right)^{g}x^{d-N}[l]. \]
If $N-d+r\geq 0$, this class is clearly positive and yields the non-emptiness of $DJ_{k,N}^{r,d}(\mu_1,\mu_2,C,l)$.  

\subsection{Transversality condition}\label{sec:nonspecial}
\indent We deal here only with the case of complete linear series $l\in\Grd$ such that $|D|=l$ for some $D\in C_d$.
Consider the alternative description of the space $DJ^{r,d}_{k,N}(\mu_1\mu_2,C,l)$ as the intersection
\[ DJ^{r,d}_{k,N}(\mu_1,\mu_2,C,l) = \Sigma_{k,N}(\mu_1,\mu_2) \cap |D|. \]
The condition for transversality of intersection is:
\begin{equation}\label{eq:transvers}
T_D(C_d) = T_D(\Sigma_{k,N}(\mu_1,\mu_2)) + T_D(|D|), 
\end{equation}
for $D=\sum_{i=1}^k a_i D_i$, for divisors $D_i \in C_{d_i}$ and fixed vectors of strictly positive integers $\mu_1=(a_1,\ldots,a_k)$ and $\mu_2=(d_1,\ldots,d_k)$ satisfying $\sum_{i=1}^k a_k d_k=d$.

Recall that $T_D (C_d) = H^0(C,\oo_D (D))$, as shown for instance in \autocite{ACGH} chapter IV, \S 1.  Moreover its dual is $T_D^{\vee}(C_d)=H^0(K_C/K_C(-D))$ and the pairing between the tangent and cotangent space is given by the residue.

To compute $T_D(\Sigma_{k,N}(\mu_1,\mu_2))$, 
let $\mathcal{D}_i$ denote the diagonal in the $a_i$-th product $C_{d_i}\times\ldots \times C_{d_i}$ so that $\Sigma_{k,N}=\mathcal{D}_1 \times\ldots\times\mathcal{D}_k/S_d$.  Hence
\[ T_D(\Sigma_{k,N}(\mu_1,\mu_2)) = T_{a_1 D_1}\mathcal{D}_1\oplus\ldots\oplus T_{a_k D_k}\mathcal{D}_k. \]
Since $\mathcal{T}\mathcal{D}_i=\mathcal{T}C_{d_i}$ for all $i=1,\ldots,k$,
\begin{align*}
T_D(\Sigma_{k,N}(\mu_1,\mu_2)) &= T_{D_1}C_{d_1} \oplus \ldots \oplus T_{D_k}C_{d_k} \\
&\simeq T_{(D_1,\ldots,D_k)} C_{d_1} \times \ldots \times C_{d_k} \\
&\simeq T_{D_1 + \ldots + D_k} C_N \\
&= H^0(C,\oo_C(D_1 + \ldots  + D_k)/\oo_C)\\
&\simeq H^0(K_C(-D_1-\ldots-D_k)/K_C(-D))^0,
\end{align*}
 where the superscript $^0$ denotes the annihilator of a vector space (or more precisely the orthogonality with respect to the natural pairing given by the residue mentioned above).  The last isomorphism follows from the following argument: consider the following exact sequence of sheaves:
\[ 0\rightarrow \oo_C(D_1 + \ldots + D_k)/\oo_C \rightarrow \oo_C(D)/\oo_C\rightarrow \oo_C(D)/\oo_C(D_1 + \ldots + D_k)\rightarrow 0.\]
We then immediately get that 
\[ H^0 (C, \oo_C(D)/\oo_C(D_1 + \ldots + D_k))= H^0(C,\oo_C(D)/\oo_C)\Big/H^0(C,\oo_C(D_1+\ldots+D_k)/\oo_C). \]
Dualising, we obtain
\[ H^0(C, \oo_C(D)/\oo_C(D_1 + \ldots + D_k))^{\vee}\simeq H^0(C,\oo_C(D_1 + \ldots  + D_k)/\oo_C)^{0}. \]
Furthermore, Serre duality applied to the long exact cohomology sequence of the sequence 
\[ 0\rightarrow \oo_C(D_1 + \ldots + D_k) \rightarrow \oo_C(D) \rightarrow \oo_C(D)/\oo_C(D_1 + \ldots + D_k)\rightarrow 0 \]
yields
\[ H^0(C, \oo_C(D)/\oo_C(D_1 + \ldots + D_k))^{\vee} \simeq H^0(K_C(-D_1-\ldots-D_k)/K_C(-D)). \]
Thus, we conclude that
\[T_D(\Sigma_{k,N}(\mu_1,\mu_2))\simeq H^0(K_C(-D_1-\ldots-D_k)/K_C(-D))^{0}.\]

To determine $T_D |D|$, consider the following restriction of the Abel-Jacobi map:
\[u:\Crd \rightarrow \Wrd\]
with differential given by
\[\delta:\im(\alpha\mu_0)^0\rightarrow\im(\mu_0)^0,\]
where $\delta$ denotes the restriction of the coboundary map 
\[H^0(C,\oo_D(D))\rightarrow H^1(C,\oo_C)\]
of the Mittag-Leffler sequence to $T_D\Crd=\im(\alpha\mu_0)^0$, while
\[\alpha:H^0(C,K_C)\rightarrow H^0(C,K_C\otimes\oo_D)\] is the restriction mapping and
\[ \mu_0:H^0(C,K_C-D)\otimes H^0(C,\oo_C(D))\rightarrow H^0(C,K_C) \]
the cup-product mapping (see Chapter IV of \autocite{ACGH} for details).

Let $D\in\Crd$.  Then $\vert D \vert\subset\Crd$ and $u(D)\in\Wrd$ with $u^{-1}(u(D))=\vert D\vert$.  Since $\delta$ is surjective by definition,
\[T_D\vert D\vert = T_D (u^{-1}(u(D)))=\ker\delta=\im(\alpha)^0,\]
where the dual map $\delta^{\vee}$ is the restriction of $\alpha$ to $(\im(\mu_0)^0)^{\vee}=\coker(\mu_0)$.

The transversality condition (\ref{eq:transvers}) translates to
\[T_D C_d = H^0\Bigl(C,K_C\Bigl(-\sum_{i=1}^k D_i\Bigr)/K_C(-D)\Bigr)^0 + \im(\alpha)^0\]
which is equivalent to:
\[H^0\Bigl(C,K_C\Bigl(-\sum_{i=1}^k D_i\Bigr)/K_C(-D)\Bigr)\cap\im(\alpha)=0.\]
We conclude that the transversality condition (\ref{eq:transvers}) ca be reformulated as:
\begin{equation}\label{eq:transversefinal}
H^0(C,K_C-D-D_1-\ldots-D_k)=0.
\end{equation}  


Note that the condition (\ref{eq:transversefinal}) is immediately satisfied by non-special and canonical linear series therefore proving Theorem \ref{dimension} in these cases.
There are actually a few more cases where transversality follows without using degenerations to nodal curves.   

\subsubsection{The case $r=1$}
The argument in this case is similar to the one in Section 5 of \autocite{HM}.  The idea is to consider the map 
\[ \pi:C\rightarrow\p^1 \]
given by the de Jonqui\`eres divisor $D=\sum_{i=1}^k a_i D_i$ and its versal deformation space $\mathcal{V}$.  Moreover, let $\mathcal{V}'\subset\mathcal{V}$ be the subvariety of maps given by divisors with the same coefficients as $D$.  Then the tangent space to $\mathcal{V}$ at $\pi$ is $T_\pi \mathcal{V} = H^0(C,\mathcal{N})$, where $\mathcal{N}$ is the normal sheaf of $\pi$ defined by the exact sequence
\[ 0 \rightarrow \mathcal{T}_C \rightarrow \pi^*\mathcal{T}_{\p^1}\rightarrow\mathcal{N}\rightarrow 0, \]
where $\mathcal{T}_C$ is the tangent sheaf of $C$ and $\mathcal{T}_{\p^1}$ the tangent sheaf of $\p^1$.

Consider also the forgetful map $\beta:\mathcal{V}\rightarrow\mathcal{M}_g$, with differential $\beta_*$ given by the coboundary map 
\[ H^0(C,\mathcal{N})\rightarrow H^1(	C,\mathcal{T}_C) \]
of the exact sequence above.  We now identify the tangent space to $\mathcal{V}'$ with the subspace of $H^0(C,\mathcal{N})$ of sections of $\mathcal{N}$ that vanish in a neighbourhood of the points in the support of $D_1 +\ldots+D_k$, i.e. the sections of the sheaf $\mathcal{N}'$ defined by the sequence
\[ 0\rightarrow\mathcal{T}_C\rightarrow\pi^*\mathcal{T}_{\p^1}(-(a_1-1)D_1-\ldots-(a_k-1)D_k)\rightarrow\mathcal{N}'\rightarrow 0. \]
Since $\pi$ is a point in the general fibre of $\beta|_{\mathcal{V}'}$, from Sard's theorem it follows that the differential $\beta_*$ restricted to $\mathcal{T}_{\pi}(\mathcal{V})$ is surjective.  This in turn means that the map $\beta'$ below is surjective:
\[ H^0(C,\mathcal{N}')\xrightarrow{\beta'}H^1(C,\mathcal{T}_C)\rightarrow H^1(C,\pi^*\mathcal{T}_{\p^1}(-(a_1-1)D_1-\ldots-(a_k-1)D_k))\rightarrow 0. \]
Now, note that $\mathcal{T}_{\p^1}\simeq\oo_{\p^1}(2)$ and moreover $\pi^*\oo_{\p^1}(1)=\oo_C(D)$.  Therefore
\begin{align*}
0&=H^1(C,\oo_C(2D -(a_1-1)D_1-\ldots-(a_k-1)D_k))\\
&=H^0(C,K_C-D-D_1-\ldots-D_k)
\end{align*}
as desired.

\subsubsection{The case $r=2$}
For the case of plane curves we reformulate the transversality of intersection in terms of the dual plane $(\p^2)^{\vee}$ and closed subschemes of symmetric powers of $\p^2$.

Note first that a divisor belongs to the intersection $\Sigma_{k,N}(\mu_1,\mu_2)\cap |D|$ if and only if the points in its support are collinear.  Denote by $\mathcal{Z}^d$ the $(d+2)$-dimensional smooth subvariety of the $d$-th symmetric product of $\p^2$ corresponding to collinear length $d$ zero-cycles in $\p^2$ and let $\p^M:=\p H^0(\p^2,\oo_{\p^2}(d))$ be the space parametrising all plane curves of degree $d$, where $M=\binom{d+2}{2}-1$.  Let $\Delta$ denote the locus in $\mathcal{Z}^d$ of zero-cycles of the form $a_1 D_1 + \ldots + a_k D_k$, where the $D_i$ are collinear length $d_i$ zero-cycles, $\sum_{i=1}^k d_i=N$, and $\sum_{i=1}^k a_i d_i = d$.  We immediately get that $\Delta$ is isomorphic to $\mathcal{Z}^N$.
Consider the following morphism:
\[ \varphi: \p^M \times (\p^2)^{\vee} \rightarrow \mathcal{Z}^d\]
given by the intersection product and denote by $\Gamma$ the preimage $\varphi^{-1}(\Delta)$. 
Moreover, we have the following equivalent description of $\Gamma$:
\[\Gamma=\{ ([C],[H]) \mid H\cap C = a_1 D_1 + \ldots + a_k D_k \}\subset
\p^M \times (\p^2)^{\vee}.\]
Note that $([C],[H])\in\Gamma$ if and only if the line $H$ is determined by an element of $\Delta$.  Moreover, for each such element of $\Delta$, the points in the supports of the $D_1,\ldots,D_k$ (with multiplicity) impose $d$ independent conditions on $\p^M$.  Hence we see that
\[\dim\Gamma=(M-d)+(N+2)=M+N-d+2.\]
Let $\pi_1:\Gamma\rightarrow\p^M$ be the projection onto the first factor.
From the existence of de Jonqui\`eres divisors when $N-d+2\geq 0$, we see that the restriction to $\Gamma$ of $\pi_1$ is dominant.  We conclude that for a general point in $\p^M$, the fibre is smooth and has dimension $N-d+2$.  Hence we conclude that for a general plane curve, the space of de Jonqui\`eres divisors $DJ_{k,N}^{r.d}(\mu_1,\mu_2,C,l)$ is smooth and of expected dimension.

\subsubsection{The case $g-d+r=1$}
Let $D=\sum_{i=1}^k a_i D_i$ be a de Jonqui\`eres divisor such that $l=|D|$ is a $\grd$ with $g-d+r=1$ and, as usual, let $L$ denote the corresponding line bundle.  This means that the residual linear series $K_C-l$ is an isolated divisor $E\in C_{2g-2-d}$ such that $K_C=\oo_C(D+E)$.  Consider the subspace
\[ \mathcal{P}_d = \{ L\in\pic^d(C) \mid h^1(C,L)=1 \}\subset\pic^d(C), \]
which, by the previous observation, has dimension $2g-2-d$.
Now consider the space
\[ \mathcal{Q}=\{ (E,D_1 + \ldots + D_k)\in C_{2g-d-2}\times C_N \mid \oo_C(E + a_1 D_1 + \ldots + a_k D_k) = K_C  \}. \]
Polishchuk \autocite{Po} shows that this space is smooth and such that 
\[\dim\mathcal{Q}=N+g-d-1.\]
Hence, for a general fixed isolated divisor $E\in C_{2g-d-2}$, the space
\[ \mathcal{Q}'=\{ D_1 +\ldots+ D_k\in C_N \mid \oo_C(a_1 D_1 + \ldots + a_k  D_k) = \oo_C(K_C-E) \} \]
is also smooth and of dimension
\[ (N+g-d-1)-(2g-d-2) = N-g+1=N-d+r,\]
which immediately implies the same for the space $DJ^{r,d}_{k,N}(\mu_1,\mu_2,C,l)$ for a general linear series $l$ with $g-d+r=1$.  

We can in fact do better than this and prove transversality for an arbitrary linear series $l$ with $g-d+r=1$.  From Polishchuk's result we have that the intersection
\[ \Sigma=\{ E+D \in C_{2g-2} \mid D\in \Sigma_{k,N}(\mu_1,\mu_2)\}\cap |K_C| \]
is transverse, i.e.
\[ T_{E+D} (\Sigma) + T_{E+D}|K_C| = T_{E+D} (C_{2g-2}). \]
Using the fact that
\begin{align*}
T_{E+D} (\Sigma) &= T_E (C_{2g-d-2}) \oplus T_{D} (\Sigma_{k,N})\\
T_{E+D} (C_{2g-2}) &= T_E C_{2g-d-2} \oplus T_D (C_d) \\
T_{E+D} |K_C| &= T_E |E| \oplus T_D |K_C - E| = T_D |L|,
\end{align*}   
we obtain
\[ T_D \Sigma_{k,N}(\mu_1,\mu_2) + T_D |L| = T_D C_d,\]
which is the sought after transversality condition.

Therefore, in order to prove Theorem \ref{thm:smooth}, it remains to check the transversality condition (\ref{eq:transversefinal}) for $r\geq 3$ and $g-d+r\geq 2$.  We do this using degenerations in Section \ref{sec:smooth}.

\section{De Jonqui\`eres divisors on nodal curves of compact type}\label{nodes}
In the case of nodal curves, the usual correspondence between divisors and line bundles breaks down.  Most significantly for our problem, the Abel-Jacobi map
\[ C_d \rightarrow \pic^d(C) \]  
does not make sense any more, even though the two spaces $C_d$ and $\pic^d(C)$ are still defined.
As a simple example of this failure, the sheaf of functions with one pole at one of the nodes is not a line bundle, while the sheaf of functions with two poles at the node has degree 3.
We therefore first need to make sense of the statement that a linear series on a nodal curve admits a de Jonqui\`eres divisor.  We do this in a variational setting, by considering families of smooth curves degenerating to a nodal curve and analysing what happens on the central fibre to limits of line bundles admitting de Jonqui\`eres divisors.
As mentioned in the introduction, we approach this issue from two points of view: limit linear series for central fibres of compact type in Section \ref{sec:lls} and compactified Picard schemes for stable central fibres in Section \ref{sec:compactified}.  
	\subsection{Limit linear series approach}\label{sec:lls}
	\subsubsection{Review of limit linear series}\label{sec:llsreview}
\indent We recall a few well-known facts.  Consider a smooth, projective 1-parameter family $\pi:\mathscr{X}\rightarrow B$ of curves of genus $g$ over $B=\text{Spec}(R)$ with a section, where $R$ is a discrete valuation ring with uniformising parameter $t$.  Denote by $0$ the point corresponding to the maximal ideal of $R$ and by $\eta$, $\bar{\eta}$ the the generic and geometric generic point of $B$, respectively.  Finally, let the special fibre $\mathscr{X}_0$ be a reduced nodal curve of compact type, while $\mathscr{X}_{\bar{\eta}}$ is assumed to be a smooth, irreducible curve of the same genus.

Now let $(\mathscr{L}_{\bar{\eta}},\mathscr{V}_{\bar{\eta}})$ be a $\grd$ on $\mathscr{X}_{\bar{\eta}}$.  In \cite{EH86} Eisenbud and Harris show how this gives rise to a limit series on the special fibre $\mathscr{X}_0$, after possibly replacing some nodes of $\mathscr{X}_0$ by smooth rational curves via base change.  We summarise the main ideas of their construction that are needed in the course of our paper.  After the base change, one may assume that $(\mathscr{L}_{\bar{\eta}},\mathscr{V}_{\bar{\eta}})$ comes from a linear series $(\mathscr{L}_{\eta},\mathscr{V}_{\eta})$ of type $\grd$ on $\mathscr{X}_{\eta}$, which in turn determines a $\grd$ on each irreducible component $Y$ of $\mathscr{X}_0$ as follows: because $\mathscr{X}$ is smooth, the line bundle $\mathscr{L}_{\eta}$ extends to a line bundle $\mathscr{L}$ on $\mathscr{X}$.  The extension line bundle is not unique, as  
$\mathscr{L}\otimes\oo_{\mathscr{X}}(\mathcal{C})$
is also an extension of $\mathscr{L}_{\eta}$, for any Cartier divisor $\mathcal{C}$ supported on $\mathscr{X}_0$.
Hence there exists an extension $\mathscr{L}_Y$ of $\mathscr{L}_{\eta}$, unique up to isomorphism, with the property that $\deg(\mathscr{L}_Y|_Y)=d$ and for any other irreducible component $Z\neq Y$ of $\mathscr{X}_0$, $\deg(\mathscr{L}_Y|_Z)=0$.  By setting $\mathscr{V}_Y:=(\mathscr{V}_{\eta}\cap\pi_*\mathscr{L}_Y)\otimes k(0)$, we get that
\[ \mathscr{V}_Y \simeq \pi_*\mathscr{L}_Y \otimes k(0) \subseteq H^0(\mathscr{L}_Y|_{\mathscr{X}_0}) \]
is a vector space of dimension $r+1$ which we identify with its image inside $H^0(\mathscr{L}_Y|_Y)$.  The pair $(\mathscr{L}_Y|_Y,\mathscr{V}_Y)$ is thus a $\grd$ on $Y$ and is called the $Y$-\textit{aspect} of $(\mathscr{L}_{\eta},\mathscr{V}_{\eta})$.  We call the \textit{limit} of $(\mathscr{L}_{\eta},\mathscr{V}_{\eta})$ the collection of aspects
\[ l=\{ (\mathscr{L}_Y|_Y,\mathscr{V}_Y) \mid Y \text{ component of }\mathscr{X}_0 \}. \]
For a nodal curve $X$ of compact type, a collection
\[l=\{ l_Y \in G^r_d(Y) \mid Y \text{ component of }X \}\]
together with certain compatibility conditions on the \textit{vanishing sequence}
\[ 0 \leq a_0(l_Y,p) < a_1(l_Y,p) < \cdots < a_r(l_Y,p)\leq d \]
 at a point $p\in Y$, where the $a_i(l_Y,p)$ are the orders with which non-zero sections of $l_Y$ vanish at $p$, is called a \textit{crude limit linear series}.  The compatibility conditions are: if
$Z$ is another component of $\mathscr{X}_0$ such that $Y\cap Z = p$, then for all $i=0,\ldots,r$,
\begin{equation}\label{eq:vanish}
 a_i(l_Y,p) + a_{r-i}(l_Z,p) \geq d. 
\end{equation}
If we have equality in (\ref{eq:vanish}), then $l$ is called a \textit{refined limit linear series}.  Since refined limit linear series are in fact the ones playing the role of ordinary limit series on smooth curves, we shall usually drop the adjective ``refined'' unless necessary. 
It was proved in \cite{EH86} that limit linear series indeed arise as limits of ordinary linear series on smooth curves.        

Unfortunately, it is not always true that a limit linear series on a nodal curve $\mathscr{X}_0$ occurs as the limit of linear series on a family $\mathscr{X}$ of smooth curves specialising to $\mathscr{X}_0$.  While on the one hand there are examples of limit linear series that cannot be smoothed, on the other there are techniques for proving the smoothability of certain series, under some assumptions.  For details, see Section 3 of \cite{EH86}.

In the subsequent sections we shall also need the concept of a \textit{ramification sequence} at a point $p\in Y$:
\[ 0 \leq \alpha_0(l_Y,p) \leq \alpha_1(l_Y,p) \leq \cdots \leq \alpha_r(l_Y,p)\leq d-r,  \]
where $\alpha_i(l_Y,p)=a_i(l_Y,p)-i$.

\subsubsection{De Jonquières divisors on a single nodal curve of compact type}
We now make precise what we mean by saying that a limit linear series has a de Jonqui\`eres divisor.  We use the notation of the previous section.

\begin{defi}[De Jonquières divisors on a nodal curve of compact type]\label{def:comtype}
Let $X$ be a nodal curve of compact type equipped with a refined limit linear series $l$ of type $\grd$. 
Fix an integer $k\leq d$ and two vectors of positive integers $\mu_1=(a_1,\ldots,a_k)$ and $\mu_2=(d_1,\ldots,d_k)$ such that $\sum_{i=1}^{k} a_i d_i =  d$.
The divisor $\sum_{i=1}^k a_i D_i$ with $D_i\in C_{d_i}$ on $X$ is a \textit{de Jonquières divisor for} $l$ if for each irreducible component $Y\subset X$, its corresponding $\grd$ $l_Y$ has a section vanishing on $\sum_{D_{i,Y}\subset Y} a_i D_{i,Y}$, where $D_{i,Y}$ denotes the restriction of the divisor $D_i$ on the component $Y$.  We denote the space of de Jonqui\`eres divisors for a limit linear series $l$ on $X$ by $DJ^{r,d}_{k,N}(\mu_1,\mu_2,X,l)$.
\end{defi}

The sections above will also vanish at the nodes of $X$ belonging to $Y$, and in such a way that equality in (\ref{eq:vanish}) is satisfied.
Hence we can give an equivalent description for de Jonqui\`eres divisors corresponding to a refined limit linear series $l$ on $X$ which will be useful later in the paper.  We say that  $\sum_{i=1}^k a_i D_i$ is a de Jonqui\`eres divisor for $l$ if:
\begin{itemize}
	\item for each irreducible component $Y\subset X$ with only one node $q\in Y$, the series $l_Y$ admits the de Jonqui\`eres divisor
 	\[ \sum_{i=1}^k a_i D_{i,Y} + \biggl(d-\sum_{i=1}^k a_i d_{i,Y}\biggr) q, \]
 	where $d_{i,Y}=\deg D_{i,Y}$;
 	\item for each irreducible component $Y\subset X$ with at least two nodes, the series $l_Y$ admits the de Jonqui\`eres divisor
 	\[\sum_{i=1}^k a_i D_{i,Y} + \sum_{q\in\text{Sing}(\mathscr{X}_0),q\in Y}\biggl( \sum_{i=1}^k a_i d_{i,Z_q} \biggr)q,\]
 	where $Z_q$ is the irreducible component of $\mathscr{X}_0$ attached to $Y$ at the node $q\in Y$.
\end{itemize}
 
 We therefore have a way to go from de Jonquières divisors on a nodal curve of compact type to de Jonquières divisors on its smooth components, where the coefficients of the nodes must of course satisfy the equality in (\ref{eq:vanish}).

Now assume that the limit linear series $l$ in Definition \ref{def:comtype} is smoothable, that is there exists a family of curves $\pi:\mathscr{X}\rightarrow B$ and a $\grd$ $(\mathscr{L}_{\eta},\mathscr{V}_{\eta})$ as in Section \ref{sec:llsreview} whose limit is $l$.
Fix $Y\subset\mathscr{X}_0$ an irreducible component of the central fibre $\mathscr{X}_0$.  Let $\mathscr{D}_{\eta}=(\sigma)\in|\mathscr{L}_{\eta}|$ be a divisor on $\mathscr{X}_{\eta}$,  where $\sigma$ is a section of $\mathscr{L}_{\eta}$.  To find the limit of $\mathscr{D}_{\eta}$ on $\mathscr{X}_0$, we multiply $\sigma$ by the unique power of $t\in B_{\eta}$ so that it extends to a holomorphic section $\sigma_Y$ of the extension $\mathscr{L}_Y$ on the whole of $\mathscr{X}$ and so that it does not vanish identically on $\mathscr{X}_0$.  The limit of $\mathscr{D}_{\eta}$ on $\mathscr{X}_0$ is then the divisor $(\sigma_Y|_Y)$.  Thus limits of de Jonqui\`eres divisors are exactly the de Jonqui\`eres divisors for limit linear series described in Definition \ref{def:comtype}. 

In what follows we describe the space of de Jonquières divisors on families of nodal curves of compact type and endow it once more with the structure of a degeneracy locus with the help of the construction of spaces of limit linear series due to Osserman \cite{Osbook}.

\subsubsection{Degenerations of de Jonqui\`eres divisors}
Let $\mathscr{X}\rightarrow B$ be a flat, proper family of curves, where $B$ is a scheme.  Fix a partition $\mu=(a_1,\ldots,a_n)$ of $d$ and let $p_1,\ldots,p_n:B\rightarrow\mathscr{X}$ be the sections corresponding to the markings on each fibre $\mathscr{X}_t$, for $t\in B$.
By making the necessary base changes we ensure that the markings specialise on smooth points of the central fibre $\mathscr{X}_0$. 

Before discussing degenerations of de Jonquières divisors, we recall some facts about moduli spaces of (limit) linear series, following the work of Osserman \autocite{Os} and the exposition in \autocite{Osbook}.  

To begin with, let $\pi:\mathscr{X}\rightarrow B$ be a proper family of smooth curves of genus $g$ with a section.  Following Definition 4.2.1 in \autocite{Osbook}, the \textit{functor} $\mathscr{G}^r_d(\mathscr{X}/B)$ \textit{of linear series of type} $\grd$ is defined by associating to each $B$-scheme $T$ the set of equivalence classes of pairs $(\mathscr{L},\mathscr{V})$, where $\mathscr{L}$ is now a line bundle on $\mathscr{X} \times_B T$ with degree $d$ on all fibres, and $\mathscr{V}\subseteq\pi_{2*}\mathscr{L}$ is a subbundle of rank $r+1$, where $\pi_2$ denotes the second projection from the fibre product.  For the precise definition of the equivalence relation, we refer the reader to \autocite{Osbook}. This functor is represented by a scheme $G^r_d(\mathscr{X}/B)$ which is proper over $B$.

Assume now that the fibres of the family $\pi:\mathscr{X}\rightarrow B$ are nodal curves of genus $g$ of compact type such that no nodes are smoothed.  Hence all fibres have the same dual graph $\Gamma$.  For each vertex $v$ of $\Gamma$, let $Y^v_t$ denote the irreducible component of $\mathscr{X}_t$ corresponding to $v$.  Thus for each $v$ we have a family $\mathscr{Y}^v$ of smooth curves over $B$ with fibres given by $Y^v_t$.  In this case the functor $\mathscr{G}^r_d(\mathscr{X}/B)$ of linear series of type $\grd$ is defined as follows.  Consider the product fibred over $B$
\[ \prod_{v} \mathscr{G}^r_d(\mathscr{Y}^v/B). \]
Let $T$ be a scheme over $B$. A $T$-valued point of the above product consists of tuples of pairs $(\mathscr{L}^v,\mathscr{V}^v)$, where $\mathscr{L}^v$ is a vector bundle of degree $d$ on $\mathscr{Y}^v \times_B T$ and $\mathscr{V}\subseteq\pi_{2*}\mathscr{L}^v$ is a subbundle of rank $r+1$.  Denote by $\mathscr{L}^{\vec{d}}$ the ``canonical'' line bundle of degree $d$ and multidegree $\vec{d}$ on $\mathscr{X}\times_B T$ obtained as in 4.4.2 of \autocite{Osbook}.  Moreover, a line bundle has multidegree $\vec{d}^v$ if it has degree $d$ on the component corresponding to the vertex $v$ and degree zero on all the other components.  Note also that for two distinct multidegrees $\vec{d}$ and $\vec{d'}$, there is a unique twist map $f_{\vec{d},\vec{d'}}:\mathscr{L}^{\vec{d}}\rightarrow\mathscr{L}^{\vec{d'}}$ obtained by performing the unique minimal number of line bundle twists.
According to Definition 4.4.7 in loc.cit., a $T$-valued point of $\prod_{v} \mathscr{G}^r_d(\mathscr{Y}^v/B)$ is in $\mathscr{G}^r_d(\mathscr{X}/B)(T)$ if, for all multidegrees $\vec{d}$ of $d$, the map
\[ \pi_{2*} \mathscr{L}^{\vec{d}} \rightarrow \bigoplus_v (\pi_{2*}\mathscr{L}^v)/\mathscr{V}^v \] 
induced by the restriction to $\mathscr{Y}^v$ and $f_{\vec{d},\vec{d}^v}$
has its $(r+1)$st degeneracy locus equal to all of $T$.  With this construction, $\mathscr{G}^r_d(\mathscr{X}/B)$ is also represented by a scheme $G^r_d(\mathscr{X}/B)$ proper over $B$.

Finally, if $\pi:\mathscr{X}\rightarrow B$ is a smoothing family (for details, see 4.5 of \autocite{Osbook}), the irreducible components $Y^v_t$ may not exist for certain $t\in B$ and it follows that the dual graph of the fibres of the family is not constant. We assume from now on that there is a unique maximally degenerate fibre with dual graph $\Gamma_0$ (i.e.~the family is locally smoothing).  We fix a vertex $v_0 \in V(\Gamma_0)$ and set $\vec{d_0}:=\vec{d}^{v_0}$.   We then replace the tuples of pairs $(\mathscr{L}^v,\mathscr{V}^v)$ with tuples $(\mathscr{L},(\mathscr{V}^v)_{v\in V(\Gamma_0)})$, where $\mathscr{L}$ is a line bundle of multidegree $\vec{d_0}$ on $\mathscr{X}\times_B T$, and for each $v\in V(\Gamma_0)$, the $\mathscr{V}^v$ are subbundles of rank $r+1$ of the twists $\pi_{2*}\mathscr{L}^{\vec{d}^v}$.  Let $f:T\rightarrow B$ be a $B$-scheme.
 A $T$-valued point $(\mathscr{L},(\mathscr{V}^v)_{v\in V(\Gamma_0)})$ is in $\mathscr{G}^r_d(\mathscr{X}/B)(T)$ if for an open cover $\{U_m\}_{m\in I}$ of $B$ satisfying certain technical properties explained in 4.5.2 of \autocite{Osbook}, for all $m\in I$ and all multidegrees $\vec{d}$ of $d$, the map
\[ \pi_{2*}\mathscr{L}^{\vec{d}}|_{(f\circ \pi_2)^{-1}(U_m)} \rightarrow \bigoplus_v \left( \pi_{2*}\mathscr{L}^{\vec{d}^v}|_{(f\circ\pi_2)^{-1}(U_m)} \right)/\mathscr{V}^v|_{f^{-1}(U_m)},  \]
induced by the appropriate (local) twist maps,
has its $(r+1)$st degeneracy locus equal to the whole of $U_m$.

\begin{rem}\label{rem:functor}
The functor of linear series with points given by tuples $(\mathscr{L},(\mathscr{V}^v)_{v\in V(\Gamma_0)})$  is naturally isomorphic to the linear series functor with points given by tuples of pairs $(\mathscr{L}^v,\mathscr{V}^v)$ in the case of families where no nodes are smoothed (this is Proposition 4.5.5 in loc.cit.).
\end{rem}

\begin{rem}
All the constructions can be shown to be independent of the choice of vertex $v_0$, twist maps, and open covers $\{U_m\}_{m\in I}$.
\end{rem}

Note that all constructions are compatible with base change and moreover, the fibre over $t\in B$ is a limit linear series space when $\mathscr{X}_t$ is reducible, and a space of classical linear series when $\mathscr{X}_t$ is smooth.  As a last remark, since working with (refined) limit linear series in the sense of Eisenbud and Harris is more convenient for practical purposes, we generally restrict to those (see Section 6 of \autocite{Os} for the connection between the two approaches).



Denote by $\ell$ a $T$-valued point of $\mathscr{G}^r_d(\mathscr{X}/B)(T)$.
In what follows, we construct a functor $\mathcal{DJ}^{r,d}_{k,N}(\mu_1,\mu_2,\mathscr{X},\ell)$, represented by a scheme which is projective over $B$, and which parametrises de Jonquières divisors for a family $\mathscr{X}\rightarrow B$ of curves of genus $g$ of compact type equipped with a linear series $\ell$.  


\begin{prop}\label{prop:dejspacecomtype}
Fix a projective, flat family of curves $\mathscr{X}\rightarrow B$ over a scheme $B$ equipped with a linear series $\ell$ of type $\grd$.  Let $\mu_1=(a_1,\ldots,a_k)$ and $\mu_2=(d_1.\ldots,d_k)$
be vectors of positive integers such that $\sum_{i=1}^k a_i d_i = d$. As usual, let $N=\sum_{i=1}^k d_i$.  Consider also the
relative divisors $\mathscr{D}_i\subset\mathscr{X}^{d_i}$.
There exists a scheme $\mathcal{DJ}^{r,d}_{k,N}(\mu_1,\mu_2,\mathscr{X},\ell)$ projective over $B$, compatible with base change, whose point over every $t\in B$ parametrises pointed curves $[\mathscr{X}_t, \mathscr{D}_1(t),\ldots,\mathscr{D}_k(t)]$ such that $\sum_{i=1}^k a_i \mathscr{D}_i(t)$ is a de Jonqui\`eres divisor of $\ell_t$.
Furthermore, every irreducible component of $\mathcal{DJ}^{r,d}_{k,N}(\mu_1,\mu_2,\mathscr{X},\ell)$ has dimension at least $\dim B +N- d+r$.
\end{prop}
\begin{proof}
 We construct the functor $\mathcal{DJ}^{r,d}_{k,N}(\mu_1,\mu_2,\mathscr{X},\ell)$ as a subfunctor of the functor of points of the fibre product $\mathscr{X}^N$ over $B$.  We show that it is representable by a scheme that is projective over $B$ and which we also denote by $\mathcal{DJ}^{r,d}_{k,N}(\mu_1,\mu_2,\mathscr{X},\ell)$.  

Let $T\rightarrow B$ be a scheme over $B$.
Suppose first that all the fibres of the family are nonsingular.  In this case, from the discussion above, a $\grd$ on $\mathscr{X}$ is given by a pair $(\mathscr{L},\mathscr{V})$, where $\mathscr{V}\subseteq\pi_{2*}\mathscr{L}$ is a vector bundle of rank $r+1$ on $B$.  Then the $T$-valued point $[\mathscr{X},\mathscr{D}_1,\ldots,\mathscr{D}_k]$ belongs to $\mathcal{DJ}^{r,d}_{k,N}(\mu_1,\mu_2,\mathscr{X},\ell)(T)$ if the $r$-th degeneracy locus of the map
\[ \mathscr{V} \rightarrow \pi_{2*}\mathscr{L}|_{\sum_{i=1}^k a_i \mathscr{D}_i } \]
is the whole of $T$.  By construction $\mathcal{DJ}^{r,d}_{k,N}(\mu_1,\mu_2,\mathscr{X},\ell)$ is compatible with base change, so it is a functor, and it has the structure of a closed subscheme, hence it is representable and the associated scheme is projective.  

Alternatively, more explicitly, take the projective bundle $\p\mathscr{V}$ corresponding to $\mathscr{V}$ which has rank $r$, with elements in its fibres given by sections $\sigma\in H^0(\mathscr{L}|_{\mathscr{X}_t})$ up to equivalence with respect to scalar multiplication.  Consider the subscheme $\mathcal{DJ}'(\mathscr{X},\mathscr{V})$ in $\p\mathscr{V}$ cut by the equations coming from the condition that the sections vanish on $\mathscr{D}_i$ with multiplicity at least $a_i$.  This imposes in total $\sum_{i=1}^k a_i d_i = d$ conditions, so the dimension of every irreducible component of $\mathcal{DJ}'(\mathscr{X},\mathscr{V})$ is at least $\dim B +N- d + r$.
Collecting all irreducible components of $\mathcal{DJ}'(\mathscr{X},\mathscr{V})$ such that the section $\sigma$ does not vanish on the whole underlying curve, we obtain the desired $\mathcal{DJ}^{r,d}_{k,N}(\mu_1,\mu_2,\mathscr{X},\ell)$. 

Now suppose that some of the fibres have nodes (that may or may not be smoothed by $\mathscr{X}$ - see Remark \ref{rem:functor}).  From the discussion above, a $\grd$ on $\mathscr{X}$ is a tuple $(\mathscr{L},(\mathscr{V}^v)_{v\in V(\Gamma_0)})$.  Let $v_j\in \Gamma_0$ be the vertex corresponding to an irreducible component.  Denote by $\mathscr{D}_{i,j}$ the specialisation of $\mathscr{D}_i$ to $v_j$.  Then the $T$-valued point $[\mathscr{X},\mathscr{D}_1,\ldots,\mathscr{D}_k]$ belongs to $\mathcal{DJ}_{k,N}^{r,d}(\mu_1,\mu_2,\mathscr{X},\ell)(T)$ if, for all vertices $v_j$, the $r$-th degeneracy locus of the map
\[\mathscr{V}^{v_j} \rightarrow \pi_{2*}\mathscr{L}^{\vec{d^{v_j}}}|_{a_i \mathscr{D}_{i,j}}\]
is the whole of $T$.  Checking for compatibility with base change (and hence functoriality) is more delicate than in the previous case because the base change may change the graph $\Gamma_0$.  However, arguing like in the proof of Proposition 4.5.6 in loc.cit.~yields the desired property.
Representability and projectiveness then follow analogously.


Alternatively, if no nodes are smoothed in $\mathscr{X}$, for each vertex $v$ of the dual graph $\Gamma$ of the fibres, we have a family $\mathscr{Y}^v$ of smooth curves with the divisors $\mathscr{D}_i$ belonging to $\mathscr{Y}^v$ and additional sections $q_j$ corresponding to the preimages of the nodes.  
Consider now the space $\mathcal{DJ}'(\mathscr{Y}^v,\mathscr{V}^v)$ defined as in the case of families with smooth fibres by the vanishing at the $\mathscr{D}_i$.  In addition, we cut 
$\mathcal{DJ}'(\mathscr{Y}^v,\mathscr{V}^v)$ with the equations corresponding to the vanishing of the sections at the points $q_j$, subject to the constraints explained in the discussion following Definition \ref{def:comtype}.  We denote the space thus obtained by $\mathcal{DJ}'(\mathscr{Y}^v,\mathscr{V}^v)$ as well.
Finally, the desired space $\mathcal{DJ}_{k,N}^{r,d}(\mu_1,\mu_2,\mathscr{X},\ell)$ is obtained by taking the fibre product over $B$ of the $\mathcal{DJ}'(\mathscr{Y}^v,\mathscr{V}^v)$.  The dimension estimate follows as in the case of smooth fibres.
If there are smoothed nodes, for each $v\in V(\Gamma_0)$, consider the subscheme $\mathcal{DJ}'(\mathscr{X},\mathscr{V}^v)$ in $\p\mathscr{V}^v$ cut by the vanishing conditions at the divisors $\mathscr{D}_i$ and at the nodes.  Taking the fibre product over $B$ yields the space $\mathcal{DJ}_{k,N}^{r,d}(\mu_1,\mu_2,\mathscr{X},\ell)$ and the dimension bound.
\end{proof}

\begin{rem}\label{rem:fibre}
Let $\phi:\mathcal{DJ}_{k,N}^{r,d}(\mu_1,\mu_2,\mathscr{X},\ell)\rightarrow\mathscr{X}$ be the forgetful map, which is projective by base change.  Then the fibre of $\phi$ over a curve $\mathscr{X}_t$ is precisely
$DJ_{k,N}^{r,d}(\mu_1,\mu_2,\mathscr{X},\ell)$.
\end{rem}

To conclude the study of the space $\mathcal{DJ}_{k,N}^{r,d}(\mu_1,\mu_2,\mathscr{X},\ell)$ of de Jonquières divisors for a family of curves, we investigate their smoothability.

\begin{prop}
Suppose the pointed curve $[C,D_1,\ldots,D_k]\in B$ is contained in an irreducible component $U\subset \mathcal{DJ}_{k,N}^{r,d}(\mu_1,\mu_2,\mathscr{X}/B,\ell)$ with $\dim U = \dim B+N - d +r$.  Then the general point of $U$ parametrises a de Jonquières divisor on a smooth curve.
\end{prop} 
\begin{proof}
We essentially follow the argument in the proof of Theorem 3.4 of \autocite{EH86}. 

Let $\widetilde{\mathscr{X}}\rightarrow\widetilde{B}$ be the versal family of pointed curves around $[C,D_1,\ldots,D_k]$ and let $f:B\rightarrow\widetilde{B}$ be the map inducing $\pi:\mathscr{X}\rightarrow B$ with sections corresponding to the marked points.  Moreover, let $\widetilde{\mathscr{L}}$ be the corresponding linear series on $\widetilde{\mathscr{X}}$.  Let $\widetilde{U}\subset \mathcal{DJ}_{k,N}^{r,d}(\mu_1,\mu_2,\widetilde{\mathscr{X}}/\widetilde{B},\widetilde{\ell})$ be a component such that $U\subset f^*\widetilde{U}$ and denote by $\widetilde{C}$ the point of $\widetilde{U}$ corresponding to $C$.  By Proposition \ref{prop:dejspacecomtype}, $\dim\widetilde{U}\geq \dim\widetilde{B}+N - d + r$.  Hence, if $\widetilde{U}$ does not completely lie in the discriminant locus of $\widetilde{X}\rightarrow\widetilde{B}$ which parametrises nodal curves, then a general point of $\widetilde{U}$ corresponds to a de Jonquières divisor on a smooth curve.  On the other hand, if $\widetilde{U}$ lies over a component $\widetilde{B}'$ of the discriminant locus, then 
\[ \dim\widetilde{U}\geq \dim \widetilde{B}+N - d+r > \dim \widetilde{B}'+N - d+r, \]
since $\widetilde{B}'$ is a hypersurface in $\widetilde{B}$.  Therefore every component of $f^*\widetilde{U}$ (hence also $U$) must have dimension strictly larger than $\dim B+N -d+r$ which contradicts the assumption on $\dim U$.  Hence $\widetilde{U}$ cannot lie entirely in the discriminant locus, and we are done.
\end{proof}

\subsection{De Jonquières divisors on nodal stable curves}\label{sec:compactified}
\indent Consider now a smooth 1-parameter family $\pi:\mathscr{X}\rightarrow B$ of curves of genus $g$ over the one-dimensional scheme $B$ such that the fibres over $B^*=B\setminus 0$ are smooth curves, while the special fibre is given by a stable nodal curve $\mathscr{X}_0$.   
Denote by $I(\mathscr{X}_0)$ the set of all irreducible components of the central fibre and by $N(\mathscr{X}_0)$ the set of nodes lying at the intersection of distinct irreducible components, together with their respective supports, i.e.
\[ N(\mathscr{X}_0) = \{ (q,C) \mid q\in C\cap C'\text{ where }C,C'\in I(\mathscr{X}_0) \}. \]
 Suppose that $\mathscr{L}^*$ is a line bundle on $\mathscr{X}^*$ such that the restriction $\mathscr{L}_t$ to each fibre $\mathscr{X}_t$ is of degree $d$ for all $t\in B^*$.  Then, using Caporaso's approach \autocite{Ca} we can extend $\mathscr{L}^*$ over the central fibre $0\in B$ such that the fibre $\mathscr{L}_0$ is a limit line bundle on $\mathscr{X}_0$ (or possibly a quasistable curve of $\mathscr{X}_0$) of degree $d$.   
As observed before, this limit is not unique because, for any $m_C\in\mathbb{Z}$,
\begin{equation}\label{def:twistbundle}
\mathscr{L}\otimes \oo_{\mathscr{X}}\Bigl(\sum_{C\in I(\mathscr{X}_0)} m_C C \Bigr)\end{equation} is also an extension of $\mathscr{L}^*$ to $B$.  We call the new extension in (\ref{def:twistbundle}) a \textit{twisted line bundle}.  Observe also the following ``computation'' rules
\begin{align}\label{eq:rules}
\begin{split}
\oo_{\mathscr{X}}&\simeq\oo_{\mathscr{X}}\Bigl(\sum_{C\in I(\mathscr{X}_0)} C\Bigr)\\
\oo_{\mathscr{X}}\Bigl(\sum_{C\in I(\mathscr{X}_0)} m_ C C\Bigr)\Bigl|_{C'}&\simeq\oo_{C'}\Bigl(\sum_{q\in C\cap C'} (m_C-m_{C'}) q\Bigr).
\end{split}
\end{align}

We encode this information in a \textit{twist function}:
\begin{align*}
T: N(\mathscr{X}_0) &\rightarrow \mathbb{Z} \\
(q,C) &\mapsto m_{C'} - m_C
\end{align*}
and introduce the following 
\begin{defi}\label{def:twist}
A \textit{twist} of the line bundle $\mathscr{L}$ is a function $T: N(\mathscr{X}_0) \rightarrow \mathbb{Z}$ satisfying the following properties
\begin{enumerate}
	\item Given $C,C'\in I(\mathscr{X}_0)$ and $q\in C\cap C'$, then $T(q,C)=-T(q,C')$.
	\item Given $C,C'\in I(\mathscr{X}_0)$ and $q_1,\ldots,q_n\in C\cap C'$, then 
	\[T(q_1,C)=\ldots=T(q_n,C) = -T(q_1,C')=\ldots=-T(q_n,C).\]
	\item Given $C,C',\widehat{C},\widehat{C'}\in I(\mathscr{X}_0)$, and points $q_C \in C\cap \widehat{C}$, $q_{C'}\in C'\cap\widehat{C'}$, $q\in C\cap C'$, and $\widehat{q}\in \widehat{C}\cap\widehat{C'}$, such that
	\[  T(q_C,C)=T(q_{C'},C')=0, \]
	we have that
	\[ T(q,C)=T(\widehat{q},\widehat{C}). \]
\end{enumerate}
\end{defi}

\begin{rem}
The definition for the twist $T$ of a line bundle $L$ on a single curve $X$ is analogous.
\end{rem}

Let $\underline{d}$ denote the \textit{multidegree} $\underline{d}=(d_C)_{C\in I(\mathscr{X}_0)}$ of $\mathscr{L}_0$ assigning to each irreducible component $C \in I(\mathscr{X}_0)$ the degree of $\mathscr{L}_0|_{C_j}$.    

Since we ultimately want to degenerate de Jonqui\`eres divisors, we introduce markings: 
let the sections $p_1,\ldots,p_n:B\rightarrow\mathscr{X}$ correspond to the $n$ markings on each of the fibres $\mathscr{X}_t$.  
 
We now specify what happens to the limit line bundle $\mathscr{L}_0$.  We rely on Melo's construction of the compactified Picard stack on the moduli stack of curves with marked points described in \autocite{Me}.  The Caporaso compactification emerges as a particular case (where  no markings are present). We chose to work with this compactification (instead of using rank-1 torsion-free sheaves) because we want to use an induction procedure involving restrictions of line bundles on different irreducible components of the nodal curve, as in the arguments from \ref{sec:step1} and \ref{sec:step2}.  Rank-1 torsion-free sheaves would not allow this, since their restrictions to subcurves are not necessarily torsion-free themselves.

We summarise here the most important aspects of this compactification that are relevant to us. Let $X$ be a semistable curve of genus $g\geq 2$ with $n$ marked points.  For a subcurve $X'\subset X$, let $k_{X'} = \#\Bigl( X' \cap \overline{X\setminus X'} \Bigr)$.  A \textit{rational tail} $C$ of $X$ is a rational proper subcurve with $k_{X'} =1$, whereas a \textit{rational bridge} is a rational proper subcurve $X'$ of $X$ satisfying $k_{X'} =2$.  An \textit{exceptional component} of $X$ is a destabilising component without marked points.  Finally the semi-stable curve $X$ is called  \textit{quasi-stable} if the following conditions are satisfied:
\begin{itemize}
	\item all destabilising  components are exceptional;
	\item rational tails do not contain any exceptional components;
	\item each rational bridge contains at most one exceptional component.
\end{itemize}

\begin{defi}
Let $Y$ be a quasi-stable curve (obtained via semi-stable reduction) of the stable curve $X$ of genus $g\geq 2$ with $n$ marked points equipped with a line bunde $L$ of degree $d$.  The multidegree of $L$ is \textit{balanced} if
\begin{enumerate}
	\item If $Y'\subset Y$ is an exceptional component, then $\deg_{Y'} L=1$.
	\item If $Y'$ is a rational bridge, then $\deg_{Y'} L\in\{0,1\}$.
	\item If $Y'$ is a rational tail, then $\deg_{Y'} L=-1$.
	\item If $Y'$ is a proper subcurve whose irreducible components are not contained in any rational tail or bridge, then $\deg_{Y'} L$ must satisfy the following inequality:
	\begin{equation}\label{ineq}\Bigl| \deg_{Y'} L - \frac{d(w_{Y'} - t_{Y'})}{2g-2} -t_{Y'} \Bigr| \leq \frac{k_{Y'} - t_{Y'} - 2b^L_{Y'}}{2},
	\end{equation}
	where $w_{Y'} = 2(g_{Y'} - 2)$, $t_{Y'}$ is the number of rational tails meeting $Y'$, and $b^L_{Y'}$ is the number of rational bridges where the degree of $L$ vanishes and which meet $Y'$ in two points.
\end{enumerate}
Denote by $\overline{P}^X_{d}$ the set of all the pairs $(Y,L)$ of quasi-stable curves $Y$ of $X$ equipped with a balanced line bundle $L$ of degree $d$.  Let $\overline{W}^X_{r,d}\subset\overline{P}^X_{d}$ denote all those pairs where the line bundles satisfy $h^0(Y,L)\geq r+1$.
\end{defi}

The compactification $\overline{P}_{d,g,n}$ of the Picard stack on the moduli stack of stable curves with marked points is given by the line bundles with balanced multidegrees on quasistable curves. 
More pecisely, $\overline{P}_{d,g,n}$ is a smooth and irreducible Artin stack of dimension $4g-3+n$ whose objects over a scheme $B$ are families $(\pi:\mathscr{X}\rightarrow B, p_i: B\rightarrow\mathscr{X}, \mathscr{L})$ of quasi-stable curves of genus $g$ with $n$ marked points equipped with a relative degree $d$ balanced line bundle $\mathscr{L}$.  The stack $\overline{P}_{d,g,n}$ is endowed with a (forgetful) universally closed morphism $\Psi_{d,g,n}$ onto $\overline{\mathcal{M}}_{g,n}$.
If moreover $(d-g+1,2g-2)=1$, the rigidification (in the sense of \autocite{ACV}) of $\overline{P}_{d,g,n}$ is a Deligne-Mumford stack and the morphism $\Psi_{d,g,n}$ is proper. 
For more details, see \autocite{Me} Definition 4.1, Theorem 4.2, and Section 7.


In what follows we also need the result below (for a proof, see \cite{Ra} Proposition 6.1.3).

\begin{lemma}\label{lemma:raynaud}
Let $B$ be a smooth curve and let $f:\mathscr{X}\rightarrow B$ be a flat and proper morphism.  Fix a point $b_0\in B$ and set $B^*=B\setminus b_0$.  Let $\mathscr{L}$ and $\mathscr{M}$ be two line bundles on $\mathscr{X}$ such that $\mathscr{L}|_{f^{-1}(B^*)}\simeq\mathscr{M}_{f^{-1}(B^*)}$.  Then
\[  \mathscr{L} = \mathscr{M} \otimes \oo_{\mathscr{X}}(\mathcal{C}), \]
where $\mathcal{C}$ is a Cartier divisor on $\mathscr{X}$ supported on $f^{-1}(b_0)$. 
\end{lemma}
 

With all this in mind, we define the notion of de Jonqui\`eres divisors for quasi-stable nodal curves.

\begin{defi}\label{def:dejnodes}
Fix a quasi-stable curve $Y$ of a stable curve $X$ with $n$ marked points $p_1,\ldots,p_n$. 
The line bundle $L$ with balanced multidegree $\underline{d}$ on $Y$ admits a de Jonquières divisor $\sum_{i=1}^n a_ip_i$ if there exists a twist $T$ such that, for all $C\in I(Y)$,
\[ L|_C  = \oo_{C}\Bigl(\sum_{p_i \in C} a_ip_i\Bigr) \otimes \oo_{C}\Bigl(\sum_{q\in C} T(q,C)q \Bigr).  \]
In other words, each restriction of $L$ to the irreducible components $C$ of $Y$ admits the de Jonqui\`eres divisor 
\[ \sum_{p_i\in C}a_i p_i + \sum_{q\in C} T(q,C)q. \]
\end{defi}

\begin{rem}
If $C$ is an exceptional component, then the de Jonquières divisor has only the nodes $q$ in the support.
\end{rem}

\begin{rem}
If any of the coefficients in the divisor above are negative, we find ourselves in the situation described in Section \ref{sec:lorentz}.
\end{rem}

\begin{rem}
Here, our perspective on de Jonquières divisors on quasi-stable curves is naive in the sense that we ignore the precise vanishing or residue conditions at the nodes. In what follows we construct a space that not only contains the closure of the space of smooth curves with marked points and line bundles admitting de Jonquières divisors, but also some ``virtual'' components which we keep, in the same vein as the space of twisted canonical divisor of \autocite{FP}.
\end{rem}

We now define the notion of de Jonquières divisors for a family of stable curves with $n$ marked points.  We work locally so that a Poincaré bundle exists (otherwise we would have to assume that $(d-g+1,2g-2)=1$).

\begin{defi}\label{def:relativedej}
Let $(\pi:\mathscr{X}\rightarrow B, p_i: B\rightarrow\mathscr{X}, \mathscr{L})$ be a flat, proper family of quasi-stable curves of genus $g$ with $n$ marked points equipped with a relative degree $d$ balanced line bundle $\mathscr{L}$
such that $\mathscr{L}_t\in\overline{W}^{\mathscr{X}_t}_{r,d}$.  For a fixed partition $\mu=(a_1,\ldots,a_n)$ of $d$ we say that $\mathscr{L}$ admits the de Jonqui\`eres divisor $\sum_{i=1}^n a_ip_i$ if for all $t\in B$, $\mathscr{L}_t$ admits the de Jonquières divisor $\sum_{i=1}^n a_ip_i(t)$. 
Furthermore, define the locus $\mathcal{DJ}_{g,n,\mu}^{r,d}(B)$ of de Jonquières divisors in $\overline{P}_{d,g,n}$ by
\[\mathcal{DJ}_{g,n,\mu}^{r,d}(B)=\Bigl\{ \bigl(\pi:\mathscr{X}\rightarrow B, p_i: B\rightarrow\mathscr{X}, \mathscr{L}\bigr) \mid   \mathscr{L} \text{ admits the divisor }\sum_{i=1}^n a_i p_i\Bigr\}.\]
\end{defi}

The content of the following proposition is that, for a one-parameter family of quasi-stable curves, the limit of de Jonqui\`eres divisors is itself a de Jonqui\`eres divisor.

\begin{prop}\label{prop:closed}
The locus $\mathcal{DJ}_{g,n,\mu}^{r,d}(B)$ is closed in $\overline{P}_{d,g,n}$.
\end{prop}
\begin{proof}
We use the valuative criterion.  Take a map $\iota$ from $B^*$ to $\mathcal{DJ}_{g,n,\mu}^{r,d}(B)$.  We must show that there exists a lift $\bar{\iota}$ of $\iota$ from $B$, as shown in the commutative diagram below.

\begin{figure}[H]\centering
  \begin{tikzpicture}
    \matrix (m) [matrix of math nodes,row sep=2em,column sep=1em,minimum width=1em]
  {
     B^* & \mathcal{DJ}_{g,n,\mu}^{r,d}(B) \\
     B &  \\};
  \path[-stealth]
	(m-1-1) edge [right hook ->] (m-2-1)
    (m-1-1) edge node [above] {$\iota$}  (m-1-2)
    (m-2-1) edge [dashed,->] node [above] {$\bar{\iota}$} (m-1-2);
 \end{tikzpicture}
 \end{figure}
\noindent Since a map from $B^*$ to $\mathcal{DJ}_{g,n,\mu}^{r,d}(B)$ is the same as a family  $(\pi:\mathscr{X}^*\rightarrow B^*, p_i: B^*\rightarrow\mathscr{X}^*, \mathscr{L}^*\bigr)$,  we must show that we can extend this to a family $(\pi:\mathscr{X}\rightarrow B, p_i: B\rightarrow\mathscr{X}, \mathscr{L}\bigr)$ in $\mathcal{DJ}_{g,n,\mu}^{r,d}(B)$.
In other words, we must show that if the general fibre $(\mathscr{X}_t,p_i(t),\mathscr{L}_t)$, for $t\in B^*$, is such that $\mathscr{L}_t$ admits the de Jonquières divisor $\sum_{i=1}^n a_i p_i(t)$, then the central fibre $(\mathscr{X}_0,p_i(0),\mathscr{L}_0)$ is such that $\mathscr{L}_0$ also admits the de Jonquières divisor $\sum_{i=0}^n a_i p_i(0)$.  

From Definition \ref{def:dejnodes}, the family admits de Jonquières divisors if there exists a twist $T_t$ for each fibre $\mathscr{X}_t$, with $t\in B^*$, such that, for all components $C\in\mathscr{X}_t$ and all nodes $q\in C$,
\[ \mathscr{L}_t|_C \simeq \oo_C \Bigl( \sum_{p_i(t)\in C} a_i p_i(t) \Bigr) \otimes \oo_C\Bigl( \sum_{q\in C}T_t(q,C)q \Bigr). \]
By shrinking $B$, and after possibly performing a base change, we may assume that the fibres of $\mathscr{X}$ are of constant topological type, the twist $T_t$ is the same twist $T$ over $B^*$, and there is no monodromy in the components of the fibres over $B^*$.  We must now assign a twist $T_0$ to the central fibre $\mathscr{X}_0$ equipped with $\mathscr{L}_0$.

Recall that the twist $T_0$ is a function $T_0:N(\mathscr{X}_0)\rightarrow\mathbb{Z}$.  There are two types of elements in $N(\mathscr{X}_0)$:
\begin{itemize}
	\item  $(q_0,C_0)$ where $q_0$ is a node not smoothed by the family $\mathscr{Y}$.  Here $T_0(q_0,C_0)=T(q_t,C_t)$, where $q_t$ is the corresponding node in the component $C_t$ in the generic fibre over $t\in B$.
	\item $(q_0,C_0)$ where $q_0$ is smoothed by the family $\mathscr{X}$.  Here the twist $T_0$ must be assigned ``by hand''.  
\end{itemize}	
	To do so, note also that the component $C_0\in I(\mathscr{X}_0)$ belongs to a connected subcurve $X$ of $\mathscr{X}_0$ which consists of all components belonging to the same equivalence class with respect to twists at the non-smoothed nodes, i.e.
	\[ C_0, C'_0\in X \Leftrightarrow C_0 \sim C'_0 \Leftrightarrow T(q,C_0) = T(q,C'_0) = 0, \forall q\in C_0\cap C'_0. \]
	This yields a sub-family of $\mathscr{X}\rightarrow B$, which we call $\mathscr{X}'$, whose central fibre is $X$ and whose generic fibre is given by the corresponding subcurves in $\mathscr{X}_t$.
	The markings $p_i$ which lie on the fibres of $\mathscr{X}'$ give sections which we rename $p'_i:B\rightarrow\mathscr{X}'$, for  $i=1,\ldots,n'$, where $n'\leq n$.
The nodes connecting $\mathscr{X}'_t$ to its complement in $\mathscr{X}_t$ also yield sections
$q_j:B\rightarrow\mathscr{X}'$, for $q=1,\ldots,m$, for some $m\geq 1$; we emphasize here that the $q_j(t)$ are smooth points of $\mathscr{X}'$.  Since the twist $T$ at the $q_j(t)$ is non-zero (by the definition of the equivalence classes), for $t\in B^*$ and for any component $C_t\in I(\mathscr{X}'_t)$,
\[\mathscr{L}_t|_{C_t} \simeq \oo_{C_t}\Bigl( \sum_{p'_i(t)\in C_t}a_i p'_i(t) + \sum_{q_j(t)\in C_t}T(q_j(t),C_t)q_j(t) \Bigr). \]	
	
By our previous assumptions, $T(q_j(t),C_t)$ is constant for $t\in B$, so in what follows we omit the terms in the bracket.
Hence the line bundles
\[ \mathscr{L} \text{ and } \oo_{\mathscr{X}'}\biggl(\sum_{i=1}^{n'} a_i p_i + \sum_{j=1}^m Tq_j\biggr) \]
are isomorphic over $B^*$ and they therefore differ by a Cartier divisor $\mathcal{C}$ on $\mathscr{X}'$, supported over $0\in B$.  This Cartier divisor is a sum of irreducible components of the fibre $\mathscr{X}'_0=X$, that is
\[ \mathcal{C}=\sum_{C\in I(X)} m_{C_0} C_0, \text{ with } m_{C_0}\in\mathbb{Z}. \]
Since the non-smoothed nodes of the family $\mathscr{X}'$ all have zero twist, this Cartier divisor yields in fact the definition of the twist $T_0:N(X)\rightarrow \mathbb{Z}$ for a node $q_0\in C_0\cap C'_0$ that is smoothed by $\mathscr{X}'$:
\[ (q_0,C_0) \mapsto m_{C'_0}-m_{C_0}. \]
Putting everything together, we obtain a twist $T_0:N(\mathscr{X}_0)\rightarrow\mathbb{Z}$ which by construction satisfies all the conditions of Definition \ref{def:twist}.  Moreover, for $t=0$,

\[ \mathscr{L}_0|_{C_0}  = \oo_{C_0}\biggl(\sum_{p_i \in C_0} a_ip_i(0)\biggr) \otimes \oo_{C_0}\biggl(\sum_{q_0\in C_0} T_0(q_0,C_0)q \biggr) \]
for each irreducible component $C_0$ of $\mathscr{X}_0$.  By Definition \ref{def:relativedej}, 
\[(\mathscr{X}\rightarrow B,p_i:B\rightarrow\mathscr{X},\mathscr{L})\in\mathcal{DJ}_{g,n,\mu}^{r,d}(B).\]
We conclude that $\mathcal{DJ}_{g,n,\mu}^{r,d}(B)$ is closed.
\end{proof}

\begin{rem}
Arguing like in the proof of Lemma 6 of \autocite{FP}, one can show that the line bundle associated to a de Jonquières divisor on a quasi-stable curve can be smoothed to a line bundle on a nonsingular curve.  More precisely, let $(\mathscr{X}\rightarrow B,p_i:B\rightarrow\mathscr{X},\mathscr{L})$ be a smoothing of a quasi-stable curve with marked points $[X,p_1,\ldots,p_n]$ (so $\mathscr{X}_0=X$).  Suppose also that for some $(X,L)\in\overline{P}^X_d$,
\[\oo_X\biggl( \sum_{i=1}^n a_i p_i \biggr) = L \in \overline{P}^X_d,\]
Then there exists a line bundle $\mathscr{L}'\rightarrow\mathscr{X}$ and an isomorphism $\mathscr{L}'_0\simeq L$, which is constructed by twisting $\mathscr{L}$. 
\end{rem}

For the next two results, assume that $(d-g+1,2g-2)=1$ so that the definitions of de Jonquières divisors hold not just locally, but also for families over any scheme $B$. 
We give a lower bound on the dimension of irreducible components of $\mathcal{DJ}_{g,n,\mu}^{r,d}(\overline{\mathcal{M}}_{g,n})$.

\begin{prop}
Every irreducible component of $\mathcal{DJ}_{g,n,\mu}^{r,d}(\overline{\mathcal{M}}_{g,n})$ has dimension at least $3g-3+\rho(g,r,d)+n-d+r$.
\end{prop}
\begin{proof}
The proof of this statement is the same as the one of Proposition 11 in \autocite{FP}.  The only difference is the dimension bound itself, which we explain below.

Let $[X,p_1,\ldots,p_n,L]\in\mathcal{DJ}_{g,n,\mu}^{r,d}(\overline{\mathcal{M}}_{g,n})$ and $L=\oo_X\left( \sum_{i=1}^n a_i p_i \right)$ its associated twisted line bundle.  We drop the markings $p_i$ without contracting the unstable components that we obtain.  We then add $m$ new markings to $X$ to get rid of the automorphisms of the unstable components (see loc.~cit.~for details) and we obtain a stable pointed curve $[X,q_1,\ldots,q_m]$.  Let $\mathcal{V}$ be its nonsingular versal deformation space.  Hence
\[ \dim\mathcal{V} = \dim\text{Def}([X,q_1,\ldots,q_m]) = 3g-3+m. \]
Let $\pi:\mathcal{C}\rightarrow\mathcal{V}$ be the universal curve and consider the relative moduli space $\epsilon:\mathcal{B}\rightarrow\mathcal{V}$ of line bundles of degree $d$ on the fibres of $\pi$.  Let $\mathcal{V}^*\subset\mathcal{V}$ be the locus of smooth curves and $\mathcal{B}^*\rightarrow\mathcal{V}^*$ the relative Picard scheme of degree $d$.  Finally, let $\mathcal{W}^{r*}_d\subset\mathcal{B}^*$ be the codimension at most $(r+1)(g-d+r)$ locus of line bundles with dimension of the space of sections $r+1$.  Let $\mathcal{W}^r_d$ be the closure of $\mathcal{W}^{r*}_d$ in $\mathcal{B}$.  Then 
\[  \dim\mathcal{W}^r_d \geq \dim\mathcal{B}-(r+1)(g-d+r)+m = 3g-3+\rho(g,r,d)+m. \]
This then contributes to the lower bound in the same way as in loc.~cit.
\end{proof}

Moreover, we also obtain an upper bound for the dimension of certain irreducible components of $\mathcal{DJ}_{g,n,\mu}^{r,d}(\overline{\mathcal{M}}_{g,n})$ supported on the locus of marked quasi-stable curves with at least one node.

\begin{prop}
Let $\mathcal{Z}\subset \mathcal{DJ}_{g,n,\mu}^{r,d}(\overline{M}_{g,n})$ be an irreducible component supported entirely on the locus of quasi-stable curves with $n$ marked points and at least one node.  Then $\mathcal{Z}$ has dimension at most $4g-4+n-d+r$ at a point $(X,p_1\ldots,p_n,L)$ with $L\in \overline{W}^X_{r,d}\setminus\overline{W}^X_{r+1,d}$.
\end{prop}
\begin{proof}
Denote by $\Gamma_Z$ the dual graph of the curve $X$.  By the definition of $Z$, the set $E$ of edges of $\Gamma_Z$ has at least one element.  Denote by $v$ the vertices of $\Gamma_Z$ and their set by $V$ (with $|V|\geq 1$).  By definition, each $v$ corresponds to an irreducible component of $X$ whose genus we denote by $g_v$.  Recall the genus formula:
\begin{equation}\label{eq:genusformula}
g-1 = \sum_{v\in V} (g_v-1) + |E|. 
\end{equation}
The strategy in what follows is to bound the dimension of the space of $(X,p_1,\ldots,p_n,L)\in Z$ with graph exactly $\Gamma_Z$.  Now $X$ is equipped with a line bundle $L$ of degree $d$ with strictly balanced multidegree $\underline{d}=(d_v)_{v\in\Gamma_Z}$) and $h^0(X,L)=r+1$.  Denote by $L_v$ the restriction of $L$ to the irreducible component corresponding to the vertex $v$ and by $n_v$ the number of the marked points on it.
Thus, for a fixed vertex $v$ of $\Gamma_v$ and assuming the result of Theorem \ref{thm:smooth},  the dimension of the space of de Jonquières divisors of length $n_v$ on the component corresponding to $v$ is at most $3g_v-3+\rho(g_v,r_v,d_v)+n_v-d_v+r_v$, where $r_v:=h^0(L_v)-1$.  The dimension bound is obtained by summing over the vertices
\begin{align*}
\dim Z &\leq \sum_{v\in V} \left( 3g_v - 3 + \rho(g_v,r_v,d_v) + n_v - d_v + r_v \right)\\
&\leq  3\sum_{v\in V} (g_v - 1) + n + 2|E| - d + \sum_{v\in V} r_v + \sum_{v\in V}\rho(g_v,r_v,d_v),
\end{align*}
where we used the fact that $\sum_{v\in V} n_v \leq n+2|E|$.  The surplus of $2|E|$ comes from the preimages of the nodes on each component in case the twist from the definition of de Jonquières divisors is nonzero.  From (\ref{eq:genusformula}) we have
\[ \dim Z \leq 3g-3+n-d-|E|+\sum_{v\in V}r_v + \sum_{v\in V}\rho(g_v,r_v,d_v). \]
To estimate $\sum_{v\in V}r_v$, let $X_1$ and $X_2$ be two connected subcurves of $X$ intersecting each other at $k$ nodes.  From the Mayer-Vietoris sequence
\[ 0\rightarrow H^0(X,L)\rightarrow H^0(X_1,L|_{X_1}) \oplus H^0(X_2,L|_{X_2}) \rightarrow \mathbb{C}^k \]
we obtain $h^0(X_1,L|_{X_1}) + h^0(X_2,L|_{X_2})\leq r+1+k$.  Consider in turn the same Mayer-Vietoris sequence for two connected subcurves of $X_1$ and of $X_2$, etc.~, until we are left only with irreducible components.  Working backwards and adding up the dimensions of the spaces of global sections for all irreducible components of $X$, we obtain
\begin{align*}
&\sum_{v\in V} h^0(L_v) = h^0(X,L) + |E|\Leftrightarrow \\
&\sum_{v\in V} (r_v + 1) = r+1 +|E|\Leftrightarrow\\
&\sum_{v\in V} r_v = r+1+|E|-|V|.
\end{align*}  
For the sum of Brill-Noether numbers, we use the  bound $\sum_v \rho(g_v,r_v,d_v)\leq\sum_{v\in V}g_v$, which in turn yields, using (\ref{eq:genusformula}), $\sum_{v\in V}g_v=g-1-|E|+|V|$.
Hence
$\dim Z \leq 4g-3+n-d+r-|E|\leq 4g-4+n-d+r$.  
\end{proof}

\section{The dimension theorem for complete linear series}\label{sec:prooflls}
We now give a proof of the dimension theorem (Theorem \ref{dimension}) for complete linear series (i.e.~those with $s=g-d+r\geq 0$) that makes use of the framework of limit linear series as discussed in Section \ref{sec:lls}.

We construct a nodal curve $X=C_1\cup_p C_2$ of genus $g$ out of two general pointed curves $(C_1,p)$ of genus $g_1$ and $(C_2,p)$ of genus $g_2$, where $g_1+g_2=g$.
Furthermore, we equip $X$ with a limit linear series of type $\grd$ which we construct from the corresponding aspects $\grdop{r}{d_1}(b_1p)$ on $C_1$ and $\grdop{r}{d_2}(b_2p)$ on $C_2$, where $b_1,b_2\in\mathbb{Z}_{\geq 0}$.  The genera $g_j$, the degrees $d_j$, and the multiplicities $b_j$ are chosen in such a way as to allow for a convenient induction step, where the induction hypothesis is the dimension theorem for $\grdop{r}{d_j}$ on $C_j$ for $j=1,2$.
We do this in two steps:

\begin{enumerate}
	\item The proof for series with $s\geq 2$ and $\rho(g,r,d)=0$ works by induction on $s$ (while keeping $\rho(g,r,d)=0$ fixed), with base case given by the canonical linear series on a general smooth curve (which has $s=1$ and $\rho(g,r,d)=0$).  This is done in Section \ref{sec:step1}.
	\item The proof for linear series with $\rho(g,r,d)>0$ works by induction on $\rho(g,r,d)$ (and keeping $s$ constant), with base case given by the linear series with $\rho(g,r,d)=0$ from the previous step. This is done in Section \ref{sec:step2}.
\end{enumerate}

In choosing the aspects $\grdop{r}{d_j}(b_jp)$ on $C_j$ (with $j=1,2$), one has to take the following restrictions into consideration, which ensure that the limit we constructed exists and is smoothable:
\begin{itemize}
	\item a general pointed curve $(C_j,p)\in\mathcal{M}_{g_j,1}$ may carry a $\grdop{r}{d_j}(b_j p)$ with ramification sequence at least $(\alpha_0,\ldots,\alpha_r)$ at the point $p$ if and only if (cf.~\cite{EH87}, Proposition 1.2)
	\begin{equation}\label{eq:ramif}
	\sum_{i=0}^r(\alpha_i + g_j - d +r)_+ \leq g_j, 
	\end{equation}
	where $(x)_+=\max\{x,0\}$.  In our case, the ramification sequence at $p$ is $(b_j,\ldots,b_j)$.
	\item the limit $\grd$ on $X$ must be refined in order to satisfy the hypotheses of the smoothability result of Eisenbud and Harris (Theorem 3.4 of \cite{EH83}).  This means that the inequality in (\ref{eq:vanish}) must be in fact an equality, thus further constraining the choice of $b_j$.
\end{itemize}
Combining these with Theorem 1.1 of \autocite{EH87} and Corollary 3.7 of \autocite{EH86} we obtain the smoothability of the limit $\grd$ on $X$.  Assume that we are in the setting of Definition \ref{def:comtype} with $\mathscr{X}_0=X$.  If the limit $\grd$ on $X$ admits a de Jonquières divisor $\sum_{i=1}^n a_i D_i$, then each aspect $\grdop{r}{d_j}(b_jp)$ on $C_j$ (with $j=1,2$) admits the de Jonquières divisor
\[ \sum_{i=1}^k a_i D_{i,C_j} + \Bigl(d- \sum_{i=1}^k a_i d_{i,C_j} \Bigr)p, \]
where the following inequality must hold in order to preserve the chosen ramification at $p$: 
 \begin{equation}\label{eq:ramifatp}
 d- \sum_{i=1}^k a_i d_{i,C_j} \geq b_j.
 \end{equation}
 Removing the base point $p$ from the series $\grdop{r}{d_j}(b_j p)$, we are left with a general linear series $l_j:=\grdop{r}{d_j}$ on $C_j$ (for $j=1,2$), with simple ramification at $p$ and admitting a de Jonquières divisor
\[ \sum_{i=1}^k a_i D_{i,C_j} + \Bigl(d_j- \sum_{i=1}^k a_i d_{i,C_j} \Bigr)p. \]

The strategy is to prove that $\dim DJ_{k,N}^{r,d}(\mu_1,\mu_2,X,l)\leq N-d+r$ by using the dimension theorem for the spaces of de Jonqui\`eres divisors of the series $l_j$ on $C_j$.  By the upper semicontinuity of fibre dimension applied to the map $\phi$ from Remark \ref{rem:fibre} it follows that 
\[\dim DJ_{k,N}^{r,d}(\mu_1,\mu_2,X_t,l_t)\leq N-d+r\] for a smoothing of $X$ to a general curve $X_t$ equipped with a general linear series $l_t$ of type $\grd$.  Combining this with Lemma \ref{lemma:dimbound}, we obtain the statement of the dimension theorem for a general curve with a general linear series.

\subsection{Step 1: proof for $\rho(g,r,d)=0$}\label{sec:step1}
\indent Having fixed $r$ and $s=g-d+r\geq 2$, the proof in this case works by induction on $s$.  The base case is given by the dimension theorem for the canonical linear series, (the unique linear series with index of speciality $s=1$ and vanishing Brill-Noether number), on a general smooth curve of any genus.
This follows either from our discussion in Section \ref{sec:nonspecial} or from Theorem 1.1 a) of Polishchuk \autocite{Po} with $D=0$.
The induction step constructs a curve $X$ of genus $g$ with a limit linear series $l$ of type $\grd$ with index of speciality $s$ and Brill-Noether number $\rho(g,r,d)=0$ from two irreducible components: $C_1$ equipped with a linear series $l_1$ with index of speciality $s_1=s-1$ and Brill-Noether number $\rho(l_1)=0$ and $C_2$ equipped with its canonical linear series (with index of speciality $s_2=1$).
The induction hypothesis at each step is the dimension theorem for each of the components $C_1$ and $C_2$ equipped with their respective linear series $l_1$ and $l_2$.  

We now show how to obtain the curve $X$. 
From the condition $\rho(g,r,d)=0$, we get 
\begin{align*}
&g=s(r+1),\\
&d=g+r-s.
\end{align*}
We start with a general curve $C_1$ of genus $(s-1)(r+1)$ equipped with a general linear series $l_1$ of type $\grdop{r}{g-s}$. Hence the index of speciality of $l_1$ is
\[ s_1=(s-1)(r+1)-g+s+r=(s-1)(r+1)-(s-1)(r+1)-r+s-1+r=s-1 \]
and its Brill-Noether number is
\begin{align*}
 \rho((s-1)(r+1),r,g-s)= (s-1)(r+1)-(r+1)(s-1)=0.
\end{align*} 
We choose a general point $p\in C_1$ to which we attach another general curve $C_2$ of genus $r+1$ equipped with its canonical linear series $l_2=\grdop{r}{2r}$.  This series has index of speciality $s_2=1$ and Brill-Noether number
\[ \rho(r+1,r,2r)=0. \]

Thus we obtained a curve $X=C_1\cup_p C_2$ of genus $g$.  We construct on $X$ a refined limit linear series $l$ of type $\grd$ aspect by aspect using $l_1$ and $l_2$.  On $C_1$ we take the aspect to be the series $l_1(rp)$, which therefore has the following vanishing sequence on $C_1$:
\[(r,r+1,\ldots,2r).\]  Since the limit is refined, the vanishing sequence on $C_2$ must be
\[ (d-2r,\ldots,d-r), \]
so we take the aspect corresponding to $C_2$ to be the series $l_2((d-2r)p)$.
Finally, we check that the limit series
\[ \{ (C_1,l_1(rp)),(C_2,l_2((d-2r)p) \} \]
 satisfies (\ref{eq:ramif}): 
\begin{align*}
&\text{on }C_1: \sum_{i=0}^r(r+(s-1)(r+1)-d+r)_+=(r+1)(s-1)\leq (r+1)(s-1),\\
&\text{on }C_2: \sum_{i=0}^r(d-2r+r+1-d+r)_+=r+1\leq r+1.
\end{align*}
Hence the limit linear series $l$ on $X$ is smoothable.

We now prove that $\dim DJ_{k,N}^{r,d}(\mu_1,\mu_2,X,l)\leq N-d+r$.  For $j=1,2$, let $N_j=\sum_{i=1}^k d_{i,C_j}$ and therefore $N_1 + N_2=N$.
 As seen in Section \ref{sec:lls}, $\sum_{i=1}^k a_i D_i \in DJ_{k,N}^{r,d}(\mu_1,\mu_2,X,l)$ if and only if 
\[ \sum_{i=1} a_i D_{i,C_j} +  \Bigl( d-\sum_{i=1}^k a_i d_{i,C_j} \Bigr)p\] 
is a de Jonquières divisor of length (at most) $N_j+1$ of $d$ of the aspect of $l$ corresponding to $C_j$, where $j=1,2$.  
We distinguish a few possibilities:
\begin{enumerate}
	\item If all points specialise on one of the $C_j$ (with $j=1,2$), then $d-\sum_{i=1}^k a_i d_{i,C_j}  = 0$, contradicting inequality (\ref{eq:ramifatp}).  Hence we must have
	\begin{align*}
	&d-2r\leq\sum_{i=1}^k a_i d_{i,C_1}\leq d-r, \\
	&2r\geq\sum_{i=1}^k a_i d_{i,C_2}\geq r.
	\end{align*}
	\item If $\sum_{i=1}^k a_i d_{i,C_1}=d-r$, then $\sum_{i=1}^k a_i d_{i,C_2}=r$ and moreover
	\[ \sum_{i=1}^k a_i D_{i,C_1} \in DJ_{k,N_1}^{r,d-r}(\mu'_1,\mu'_2,C_1,l_1)\text{ and } \sum_{i=1}^k a_i D_{i,C_2} + rp \in DJ_{k,N_2+1}^{r,2r}(\mu''_1,\mu''_2,C_2,l_2),\]
	where $\mu'_1=(a_i)_{D_{i,C_1}>0}$, $\mu'_2=(d_{i,C_1})$ are the strictly positive vectors
	corresponding to the component $C_1$, while
	 $\mu''_1=(a_i,r)_{D_{i,C_2}>0}$ and $\mu''_2=(d_{i,C_2},1)$ are the ones corresponding to $C_2$.
By the induction hypothesis, the following inequalities must be satisfied  
	\begin{align*}
	&\dim DJ_{k,N_1}^{r,d-r}(\mu'_1,\mu'_2,C_1,l_1)=N_1-d+2r=:x\geq 0 \\
	&\dim DJ_{k,N_2+1}^{r,2r}(\mu''_1,\mu''_2,C_2,l_2) =N_2+1-r= (N-d+r) - x +1 \geq 0,
\end{align*}		
	where we used the fact that $N_1+N_2=N$.  
Furthermore, note that on $C_2$ we are actually only interested in the locus in $DJ_{k,N_2+1}^{r,2r}(\mu''_1,\mu''_2,C_2,l_2)$ consisting of divisors with $p$ in their support.  More precisely,
consider the incidence correspondence
\[ \Gamma = \{ (D,p) \mid p\in D \} \subset DJ_{k,N_2 +1}^{r,2r}(\mu''_1,\mu''_2,C_2,l_2) \times C \]
and let $\pi_1$, $\pi_2$ be the canonical projections.  The locus we are after is $\pi_1(\pi_2^{-1}(p))$.  By construction, $\pi_2$ is dominant and since $p$ is general, 
\[\dim\pi_1(\pi_2^{-1}(p))=\dim DJ_{k,N_2 +1}^{r,2r}(\mu''_1,\mu''_2,C_2,l_2)-1.\]
	Therefore the dimension estimate for $DJ_{k,N}^{r,d}(\mu_1,\mu_2,X,l)$ is 
	\begin{align*}
	 \dim DJ_{k,N}^{r,d}(\mu_1,\mu_2,X,l)&\leq \dim DJ_{k,N_1}^{r,d-r}(\mu'_1,\mu'_2,C_1,l_1) + DJ_{k,N_2+1}^{r,2r}(\mu''_1,\mu''_2,C_2,l_2) -1\\ &= N-d+r. 
\end{align*}	
	\item If $d-2r<\sum_{i=1}^k a_i d_{i,C_1}<d-r$, then $2r>\sum_{i=1}^k a_i d_{i,C_2}>r$ and we obtain de Jonquières divisors of length $N_j+1$ on the component $C_j$, for $j=1,2$.  This yields
	\begin{align*}
	&\dim DJ_{k,N_1+1}^{r,d-r}(\mu'_1,\mu'_2,C_1,l_1)=N_1+1-d+2r=:x\geq 0 \\
	&\dim DJ_{k,N_2+1}^{r,2r}(\mu''_1,\mu''_2,C_2,l_2) =N_2+1-r= (N-d+r) - x +2 \geq 0.
\end{align*}		
	Arguing as in the previous case (for both $C_1$ and $C_2$), we obtain the same upper bound for the dimension of $DJ_{k,N}^{r,d}(\mu_1,\mu_2,X,l)$:
	\begin{align*}
	\dim DJ_{k,N}^{r,d}(\mu_1,\mu_2,X,l)&\leq \dim DJ_{k,N_1}^{r,d-r}(\mu'_1,\mu'_2,C_1,l_1)-1 + DJ_{k,N_2+1}^{r,2r}(\mu''_1,\mu''_2,C_2,l_2) -1 \\
	&= N-d+r.
\end{align*}	   
	\item If $\sum_{i=1}^k a_i d_{i,C_1}=d-2r$, then $\sum_{i=1}^k a_i d_{i,C_2}=2r$ and we get de Jonquières divisors of length $N_1+1$ on $C_1$ and of length $N_2$ on $C_2$. This case is analogous to (1) and we again obtain the upper bound $N-d+r$ for $\dim DJ_{k,N}^{r,d}(\mu_1,\mu_2,X,l)$.
\end{enumerate} 

\subsection{Step 2: proof for all $\rho(g,r,d)\geq 1$}\label{sec:step2}
\indent Fix $r$, $s=g-d+r$, and $\rho(g,r,d)\geq 1$.  We continue with the proof by induction on $\rho(g,r,d)$, where the base case is given by the dimension theorem for linear series with $\rho(g,r,d)=0$ (proved in Section \ref{sec:step1}).  The induction step constructs a curve $X$ of genus $g$ with a linear series $l$ of type $\grd$ from two components: $C_1$ equipped with a linear series $l_1$ and $C_2$ equipped with $l_2$ such that $\rho(l)=\rho(g,r,d)=\rho(l_1)+1$.  As before, the induction hypothesis at each step is the dimension theorem for the components $C_j$ and their corresponding linear series $l_j$, with $j=1,2$.  

We start with a general curve $C_1$ of genus $g-1$ equipped with a general linear series $l_1=\grdop{r}{d-1}$.  We pick a general point $p\in C_1$ and attach to it an elliptic normal curve $C_2$
with its associated linear series $l_2=\grdop{r}{r+1}$.  Note that the dimension theorem holds for the elliptic normal curve by virtue of the fact that $l_2$ is non-special (see the discussion in Section \ref{sec:nonspecial}).

The resulting curve $X=C_1\cup_p C_2$ has genus $g$
and we construct on it a limit linear series $l$ of type $\grd$ aspect by aspect.  On $C_1$ we take the aspect $\grdop{r}{d-1}(p)$, hence $p$ is a base point of the $\grd$ on $X$ with vanishing sequence on $C_1$ given by
\[ (1,2,\ldots,r+1). \]
Since the limit $\grd$ must be refined in order to be smoothable, the aspect on $C_2$ must have the following vanishing sequence at $p$
\[ (d-r-1,\ldots,d-1). \]
Thus the aspect on $C_2$ is given by the series $\grdop{r}{r+1}((d-r-1)p)$.

We check that this limit $\grd$ also satisfies (\ref{eq:ramif}):
\begin{align*}
	&\text{on }C_1: (r+1)(1+g-1-d+r)=(r+1)s\leq g-1, \\
	&\text{on }C_2: (r+1)(d-r-1+1-d+r)=0\leq 1,	 
\end{align*}
where in the first inequality we used the fact that $\rho(g,r,d)=g-(r+1)s\geq 1$.  Hence $l$ is a smoothable limit linear series on $X$.  Moreover, its Brill-Noether number is
\[  \rho(l)=\rho(g,r,d) = g-(r+1)s \]
while the linear series $l_1=\grdop{r}{d-1}$ on $C_1$ has Brill-Noether number
\[ \rho(l_1)=\rho(g-1,r,d-1) = \rho(g,r,d)-1. \]
Finally, we observe here that the induction step leaves the indices of speciality unchanged since 
$s_1=(g-1)-(d-1)+r=s$.

We now show that $\dim DJ_{k,N}^{r,d}(\mu_1,\mu_2,X,l)\leq N-d+r$.  The argument is the same as in \ref{sec:step1}.  For $j=1,2$, denote by $N_j$ the length of the divisor $\sum_{i=1}^k a_i D_{i,C_j}$.   
As for the $\rho(g,r,d)=0$, there are a few possibilities:
\begin{enumerate}
	\item If $\sum_{i=1}^k a_i d_{i,C_1}= d-1$, then $\sum_{i=1}^k a_i d_{i,C_2}=1$ and moreover
	\[ \sum_{i=1}^k a_i D_{i,C_1} \in DJ_{k,N_1}^{r,d-1}(\mu'_1,\mu'_2,C_1,l_1)\text{ and }\sum_{i=1}^k a_i D_{i,C_2} + rp\in DJ_{k,N_2+1}^{r,r+1}(\mu''_1,\mu''_2,C_2,l_2), \]
	where the vectors $\mu'_1,\mu'_2,\mu''_1,\mu''_2$ are defined as in \ref{sec:step1}.
	By the induction hypothesis,	
	\begin{align*}
	&\dim DJ_{k,N_1}^{r,d-1}(\mu'_1,\mu'_2,C_1,l_1) = N_1 - d  + 1 +r =:x\geq 0,\\
	&\dim DJ_{k,N_2+1}^{r,r+1}(\mu''_1,\mu''_2,C_2,l_{C_2}) = N_2 = (N-d+r)+1-x.
	\end{align*}
	As discussed in \ref{sec:step1}, we have the bound
	\[ \dim DJ_{k,N}^{r,d}(\mu_1,\mu_2,X,l) \leq DJ_{k,N_1}^{r,d-1}(\mu'_1,\mu'_2,C_1,l_1)+DJ_{k,N_2+1}^{r,r+1}(\mu''_1,\mu''_2,C_2,l_2)-1=N-d+r. \]
	\item If $d-r-1<\sum_{i=1}^k a_i d_{i,C_1} < d-1$, then $r+1>\sum_{i=1}^k a_{i,C_2}>1$ and we get de Jonquières divisors of length $N_1+1$ on $C_1$ and length $N_2+1$ on $C_2$.  Counting dimensions as before we obtain the upper bound $N-d+r$ for the dimension of $DJ_{k,N}^{r,d}(\mu_1,\mu_2,X,l)$.
	
	\item If $\sum_{i=1}^k a_i d_{i,C_1} = d-r-1$, then $\sum_{i=1}^k a_i d_{i,C_2}=r+1$ and we have de Jonquières divisors of length $N_1+1$ on $C_1$ and length $N_2$ on $C_2$.  We obtain the same upper bound $N-d+r$.
\end{enumerate}

\section{Smoothness of the space of de Jonqui\`eres divisors}\label{sec:smooth}
\indent We now prove Theorem \ref{thm:smooth} which states that the space $DJ^{r,d}_{k,N}(\mu_1,\mu_2,C,l)$ is smooth for a complete linear series $l$ by showing that it arises as a transverse intersection of subvarieties of the symmetric product $C_d$.  Recall that we already have the result for some cases (see Section \ref{sec:nonspecial}) and it remains to show it for $r\geq 3$ and $s\geq 2$.

From the transversality condition (\ref{eq:transversefinal}), we have to show that $H^0(C,K_C-D-D_1-\ldots-D_k)=0$. To do this, we prove that 
\[ g-(d+N)+r'<0, \]
where $r'=h^0(D+D_1+\ldots+ D_k)-1=r+n'$, for some integer $n'\geq 0$.

Suppose towards a contradiction that $n'\geq N-g+d-r$. 


Consider all flag curve degenerations $j:\overline{\mathcal{M}}_{0,g}\rightarrow\overline{\mathcal{M}}_g$ and let $\mathcal{Z}:=\overline{\mathcal{M}}_{0,g}\times_{\overline{\mathcal{M}}_g}\overline{\mathcal{C}}^{N}_g$, where $\overline{\mathcal{C}}_g=\overline{\mathcal{M}}_{g,1}$.
  Let $U\subset\mathcal{Z}$ be the closure of the divisors with $r'=r+n'$ and $n'\geq N-g+d-r$ on all curves from $\im(j)\subseteq\overline{\mathcal{M}}_g$.  By assumption, the map $X\rightarrow\overline{\mathcal{M}}_{0,g}$ is dominant, hence $\dim X\geq g-3$.  Applying Proposition 2.2 of \autocite{Fa2}, there exists a point $[\widetilde{R}:=R\cup E_1\cup\ldots\cup E_g,y_1,\ldots,y_{N}]\in U$, where $R$ is a rational spine (not necessarily smooth) and the $E_i$ are elliptic tails such that either:
\begin{enumerate}[label=(\roman*), wide, labelwidth=!, labelindent=0pt]
	\item the supports of the divisors $D_1,\ldots,D_k$ coalesce into one point, or else
	\item the supports of the divisors $D_1,\ldots,D_k$ lie on a connected subcurve $Y$ of $\widetilde{R}$ of arithmetic genus $p_a(Y)=N$ and $|Y\cap\overline{(\widetilde{R}\setminus Y)}|=1$.
\end{enumerate}
Denote by $q_1,\ldots,q_g$ the points of attachment of the elliptic tails to the rational spine.

A short computation using the Plücker formula allows us to immediately dismiss Case (i). 
We now deal with Case (ii).  By assumption, there exists a proper flat morphism $\phi:\mathscr{X}\rightarrow B$ satisfying:
\begin{enumerate}[wide, labelwidth=!, labelindent=0pt]
	\item $\mathscr{X}$ is a smooth surface, $B$ is a smooth affine curve with $0\in B$ a point such that the fibre $\mathscr{X}_0$ is a curve stably equivalent to the curve $\tilde{R}$, and the fibre $\mathscr{X}_t$ is a smooth projective curve of genus $g$ for $t\neq 0$.  Furthermore, we have the 
relative divisors $\mathscr{D}_i\in\mathscr{X}^{d_i}$ with $\mathscr{D}_i(0)=D_i$, for $i=1,\ldots,k$.
	\item Let $\mathscr{X}^*=\mathscr{X}\setminus\mathscr{X}_0$.  
	Then there exists a line bundle $\mathscr{L}^*$ on $\mathscr{X}^*$ of relative degree $d$ 
and with $\dim H^0(\mathscr{X}_t,\mathscr{L}_t)=r+1$ for $t\neq 0$.
After (possibly) performing a base change and resolving the resulting singularities, the
pair $(\mathscr{L}^*,\mathscr{V}^*:=H^0(\mathscr{X}^*,\mathscr{L}^*))$ yields a refined limit linear series $\mathfrak{m}:=\grd$ on $\tilde{R}$.  The limit linear series $\mathfrak{m}$ has moreover the property that it admits the de Jonqui\`eres divisor 
$\sum_{i=1}^k a_i \mathscr{D}_i(0)=\sum_{i=1}^k a_i D_i$.  
	\item  The line bundle $\mathscr{L}^*$ also has the following property: $\mathscr{N}^*:=\mathscr{L}^*\bigl(\sum_{i=1}^k \mathscr{D}_i\bigr)$ is another line bundle on  $\mathscr{X}^*$ of relative degree $d+N$ and with $h^0(\mathscr{X}^*_t,\mathscr{N}^*_t)=r+n'+1$.    The pair $(\mathscr{N}^*,\widetilde{\mathscr{V}}^*:=H^0(\mathscr{X}^*,\mathscr{N}^*))$ also gives a limit linear series $\mathfrak{l}:=\grdop{r+n'}{d+N}$ on $\tilde{R}$.  Furthermore, the limit linear series $\mathfrak{l}$ admits the de Jonqui\`eres divisor $\sum_{i=1}^k (a_i+1) D_i$.
\end{enumerate} 
 The situation can be reformulated as follows: for $t\neq 0$,
 \[ \dim H^0(\mathscr{X}^*_t,\mathscr{N}^*_t(-\sum_{i=1}^k \mathscr{D}_i(t))=r. \]
 Then $\mathscr{N}^*\otimes\oo_{\mathscr{X}^*}(-\sum_{i=1}^k \mathscr{D}_i(B\setminus 0))$ induces the limit linear series $\grd$ that we started with.  
	
Now, for a component $C\subset\mathscr{X}_0$, let $(\mathscr{L}_C,\mathscr{V}_C)\in G^r_d(Z)$ be the $C$-aspect $\mathfrak{m}_C$ of the limit $\mathfrak{m}=\grd$.  Then there exists a unique effective divisor $D_C\in C_N$ supported only at the points of $(C\cap \bigcup_{i=1}^{k}\mathscr{D}_i(B)\cap(C\cap\overline{\mathscr{X}_0\setminus C})$ such that the $C$-aspect of $\mathfrak{m}$ has the property that the restriction map
\[ \mathscr{V}_C\rightarrow \mathscr{V}_C|_{D_C} \] has non-trivial kernel. 
For the $C$-aspect $\mathfrak{l}_C$ of the limit $\mathfrak{l}$ the situation is analogous, but now we have an effective divisor $D'_C\in C_{d+N}$ with $D'_C\geq D_C$.  Moreover, the $C$-aspect of $\mathfrak{m}$ is of the form
\[ \mathfrak{m}_C = (\mathscr{M}_C:=\mathscr{N}_C\otimes\oo_C(-D'_C+D_C), \mathscr{W}_C\subset\widetilde{\mathscr{V}}_C\cap H^0(\mathscr{M}_C)). \]
Thus, the collection $\mathfrak{m}_Y:=\{\mathfrak{m}_C\}_{C\subset Y}$ forms a limit $\grd$ on $Y$, while the collection $\mathfrak{l}_Y:=\{\mathfrak{l}_C\}_{C\subset Y}$ forms a limit $\grdop{r+n'}{d+N}$ on $Y$.

Let $p=Y\cap\overline{(\widetilde{R}\setminus Y)}$ and $Z:=\overline{\widetilde{R}\setminus Y}$.
The vanishing sequence of the limit $\grd$ at $p$ is a subsequence of the vanishing sequence at $p$ of the limit $\grdop{r+n'}{d+N}$.  The complement of this subsequence yields another limit linear series $\grdop{n'-1}{d}$ on $Y$ (see Lemma 2.1 of \autocite{Fa2}).
We distinguish two cases:
\begin{enumerate}[label=(\Roman*), wide, labelwidth=!, labelindent=0pt]
\item $N<g$. \\
To begin with, we list two technical results that help us determine a lower bound for the ramification sequence at $p$ of the limit linear series $\grdop{n'-1}{d}$ in this case. 

\begin{lemma}[Corollary 1.6 of \autocite{EH832}]\label{lemma:EH}
Let $C\simeq \p^1$ be an irreducible component of $Z$ such that $q_j \in C$ for some $j=1,\ldots,N$, where $q_j$ is the point of attachment of the elliptic tail $E_j$ to $C$.  Let $l$ be a limit linear series on $Z$ and $C'$ be another component of $\widetilde{R}$ and $q=C\cap C'$.  If $q'$ is another point on $C$, then for all but at most one value of $i$,
\[ a_i(l_C,q') < a_i(l_{C'},q). \]
\end{lemma}

\begin{lemma}\label{lemma:1original}
Let $\{\sigma_C \mid C\subseteq Y \text{ irreducible component}\}$ be the set of compatible sections corresponding to the divisor $D+D_1+\ldots+D_k$. If $q\in C$, then $\ord_q(\sigma_C)=0$.
\end{lemma}
\begin{proof}
The proof works by induction on the components of $Y$.  By construction, the tree curve $Y$ has at least two components, so the base case is $Y=C_1\cup_{q'} C_2$.  Denote by $D_{C_1}$ and
$D_{C_2}$ the specialisations of the divisor $D+D_1+\ldots+D_k$ on the two components $C_1$ and $C_2$.  Then $\sigma_{C_1}$ vanishes on $D_{C_1}$ and $\ord_{q'}(\sigma_{C_1})=d+N-\deg D_{C_1}=\deg D_{C_2}$ and similarly $\sigma_{C_2}$ must vanish on $D_{C_2}$ and $\ord_{q'}(\sigma_{C_2})=\deg D_{C_1}$.  Hence, if $q\in C_1$ is a smooth point, then $\sigma_{C_1}(q)=0$ and if $q\in C_2$ is a smooth point, then $\sigma_{C_2}(q)=0$ and we are done.

Suppose now that $Y$ has $m$ irreducible components denoted $C_1,\ldots,C_m$ and let $D_{C_1},\ldots,D_{C_m}$ be the specialisations of the divisor $D+D_1+\ldots+D_k$ to each component.  Let $C_m \cap C_{m-1}=q'$.  Then \[\ord_{q'}(\sigma_{C_{m-1}})=d+N-\deg D_{C_1}-\ldots-\deg D_{C_{m-1}}=\deg D_{C_m}.\]  Furthermore, 
\[\ord_{q'}(\sigma_{C_{m}})=d+N-\deg C_m.\]  Thus, if $q\in Y$ is a smooth point belonging to $C_m$, then $\sigma_{C_m}(q)=0$.  If $q\in Y$ belongs to any of the components of the subcurve $C_1\cup\ldots\cup C_{m-1}$, then $\sigma_{C_j}(q)=0$ (with $j=1,\ldots,m-1$), where we used the induction hypothesis and the fact that $D_{C_1} + \ldots + D_{C_{m-1}} + (\deg C_{m})q'$ is a divisor of degree $d+N$ on the subcurve $C_1\cup\ldots\cup C_{m-1}$.
\end{proof}

Let $C\subset Z$ be the irreducible component meeting $Y$ at $p$.  Denote by $C'$ the component of $Y$ containing $p$.  Suppose first that $C$ contains at least one of the points $q_j$ of attachment of the elliptic tails.  Let $p'\in C$ be a general smooth point, which therefore has vanishing sequence 
\[  a_i((\grdop{r+n'}{d+N})_C,p') = (0,1,2,3,\ldots,r+n'). \]
By Lemma \ref{lemma:EH} with $q=p$ and $q'=p'$, the vanishing sequence at $p$ is
\[a_i((\grdop{r+n'}{d+N})_{C'},p) \geq (0,2,3,4,\ldots,r+n'+1).\]
By a similar argument,
\[a_i((\grdop{r}{d})_{C'},p) \geq (0,2,3,4,\ldots,r+1).\] 
Combining this with Lemma \ref{lemma:1original}, we get the following ramification sequence for $\grdop{n'-1}{d}$:
\[ \alpha_i((\grdop{n'-1}{d})_{C'},p) \geq (1,1,\ldots,1). \]
In fact we obtain a limit linear series $\grdop{n'-1}{d}$ on $Y$ with ramification
\begin{equation}\label{eq:ramifsequence}
\alpha_i((\grdop{n'-1}{d})_{Y},p) \geq (1,1,\ldots,1).
\end{equation}  
We check a necessary condition for such a limit series to exist (cf.~Theorem 1.1 of \autocite{EH87}): 

\begin{equation}\label{eq:ramifcond3}
\sum_{i=0}^{n'-1}\widetilde{\alpha}_i + n'(N-d+n'-1)\leq N.
\end{equation} 
Since we assumed $n'\geq N-g+d-r$ and using moreover the inequality (\ref{eq:ramifsequence}) we obtain that
\[ \sum_{i=0}^{n'-1}\widetilde{\alpha}_i + n'(N-d+n'-1) \geq (N-g+d-r)(2N-g-r). \]
Denoting by $s:=g-d+r$ and using $N > d-r$, we reformulate the necessary condition (\ref{eq:ramifcond3}) as
\[ (N-s)(N-s-r)< N \]
which is equivalent to the quadratic inequality
\[ N^2 - (2s+r+1)N + s(s+1) < 0. \]
This implies that the solution $N$ must be contained in the interval
\[ \left(\frac{2s+r+1-\sqrt{(2s+r+1)^2-4s(s+r)}}{2},\frac{2s+r+1+\sqrt{(2s+r+1)^2-4s(s+r)}}{2}\right). \]
We now show that for $s\geq 2$ and $r\geq 3$
\begin{equation}\label{eq:roots}
\frac{2s+r+1+\sqrt{(2s+r+1)^2-4s(s+r)}}{2}<d-r+1,
\end{equation}
contradicting thus the hypothesis $N-d+r\geq 1$.
To do this, first note that a simple calculation yields
\[ g\geq (r+1)s \geq 2s+r+1 \]
for $s\geq 2$ and $r\geq 3$ which in turn yields
\[ 2s+r+1\leq g-((r+1)s-2s+r+1)=g-s(r-1)+r+1. \]
Another simple calculation gives, for $r\geq 3$ and $s\geq 2$:
\[\sqrt{(2s+r+1)^2-4s(s+r)} \leq (2s+r+1)-4.\]
Putting it all together, we get a sufficient condition for the inequality (\ref{eq:roots}) to be satisfied, namely:
\[ \frac{2g-2s(r-1)+2(r+1)-4}{2} < d-r+1 \]
which is equivalent to
\[ (2-r)(s-1) < 0. \]
This is clearly satisfied for $r\geq 3$ and $s\geq 2$ which means (\ref{eq:roots}) is also satisfied for these value ranges of $r$ and $s$.

\item $N\geq g$. \\
In this case $Y=\widetilde{R}$ and we check the necessary condition for the existence of a linear series $\grdop{n'-1}{d}$ on the tree curve $Y$ without specified ramification at a point (also Theorem 1.1 of \autocite{EH87}):
\begin{equation}\label{eq:ramifcond2}
n'(g-d+n'-1)\leq g.
\end{equation} 
By our assumptions, $n'\geq N-g+d-r\geq d-r$ and we therefore have
\[ \sum_{i=0}^{n'-1}\widetilde{\alpha}_i + n'(N-d+n'-1) \geq (d-r)(g-r-1). \]
Thus a necessary condition for (\ref{eq:ramifcond2}) is that
\[ (d-r)(g-r-1)\leq g, \]
which is equivalent to
\[ g\leq \frac{(r+1)(d-r)}{d-r-1}. \]
However, we also know that $s=g-d+r\geq 2$ and $g\leq\frac{r+1}{r}(d-r)$, 
which immediately give $d\geq 3r$.  This in turn yields
\[g\leq \frac{(r+1)(d-r)}{d-r-1}\leq \frac{r}{2r-1}(d-r)\leq d-r,\]
which contradicts the assumption that $s=g-d+r\geq 2$.
\end{enumerate}

\section{Non-existence statement for non-complete linear series}\label{sec:noncomplete}
In this section we prove Theorem \ref{cor:nonex} which states that, for a smooth general curve of genus $g$, if $n-d+r<0$, the general linear series $\grd$ with $g-d+r<0$ does not admit de Jonqui\`eres divisors of length $n$ of the type
\[ a_1 p_1 + \ldots + a_n p_n, \]
where the points $p_i$ in the support are distinct.  Recall that in this case, we need only one partition $\mu=(a_1,\ldots,a_n)$ of $d$ and we denote the space of de Jonqui\`eres divisors by $DJ^{r,d}_n(\mu,C,l)$.
We proceed by induction.  The base case is given by the non-existence statement in the case $n-d+r<0$ and $n<g$ shown in Lemma \ref{lemma:nonspecial}.  In the induction step we prove non-existence for $n\geq g$.

Consider the following quasi-stable curve $Y$ of genus $g\geq 4$ with $n\geq g$ marked points consisting of a smooth general curve $C$ of genus $g-1$ and a rational bridge with $n+1$ rational components $\gamma_j$, for $j=1,\ldots,n+1$. 
Since the curve is quasi-stable, at most one of the components of the rational chain is exceptional (i.e. it contains no marks).  In our case, since we have $n$ marks, there must be one such component which we denote by $\gamma_{j'}$, while each of the other rational components $\gamma_j$ contains one of the marked points $p_i$.
Let $C\cap \gamma_1=q_1$, $C\cap\gamma_{n+1}=q_{n+2}$, and $\gamma_j\cap \gamma_{j+1}=q_{j+1}$ for $j=1,\ldots,n+1$.
 The curve $Y$ is equipped with a linear series $l=\grd=(L,V)$ with $g-d+r<0$ corresponding to a line bundle $L$ with $h^0(Y,L)>r+1$.  The bundle $L$ has balanced multidegree $\underline{d}$, meaning that $\deg L_C = d-1$ and $\deg L_{\gamma_j}=0$ for all $j\neq j'$ and $\deg L_{\gamma_{j'}}=1$.  An easy Mayer-Vietoris sequence calculation yields that $C$ is also equipped with a non-complete linear series $l_C=\grdop{r}{d-1}$.  
 
This configuration gives a de Jonquières divisor on $Y$ corresponding to $l=(L,V)$ if there exists a twist $T$ satisfying the following system of linear equations:
 \begin{align*}
 	&T(q_1,C) + T(q_{n+2},C) = d-1 \\
 	&T(q_j,\gamma_j) + T(q_{j+1},\gamma_j) + \sum_{p_i\in\gamma_{j}}a_i = 0 \text{ for all }j\neq j' \\
 	&T(q_{j'},\gamma_{j'}) + T(q_{j'+1},\gamma_{j'}) = 1.
 \end{align*} 
 
Note that at least one of the terms $T(q_1,C)$ and $T(q_{n+2},C)$ must be non-zero.
There are therefore two possibilities for solutions of this system:
\begin{enumerate}
	\item Both $T(q_1,C)$ and $T(q_{n+1},C)$ are non-zero.  In this case we have a de Jonquières divisor $T(q_1,C)q_1 + T(q_{n+2},C)q_{n+2}$ on $C$ of length $2$ corresponding to $l_{C}$.  Note that since $2<g-1$ and $2-(d-1)+r=3-d+r<n-d+r<0$, the induction hypothesis yields that $l_C$ admits no such de Jonqui\`eres divisors. 
	\item Only one of the two terms is non-zero.  We then have a de Jonqui\`eres divisor of length 1 corresponding to $l_C$.  Since $1<g-1$ and $1-(d-1)+r<n-d+r<0$, the induction hypothesis yields that $l_C$ does not admit such de Jonqui\`eres divisors. 
\end{enumerate}
 
Hence $l$ does not admit any de Jonqui\`eres divisors on $Y$ of length $n\geq g$.  We now explain how to conclude the non-existence statement for a general smooth curve with a general linear series of type $\grd$.

First note that $Y$ is embedded in $\p^r$ by $l$ and using the methods of Hartshorne-Hirschowitz and Sernesi \autocite{Se} (for the precise details, see for example Lemma 1.5 of \cite{FAO}) one shows that it is flatly smoothable to a general curve of genus $g$ and degree $d$ in $\p^r$.  Thus we have a family $\pi:\mathscr{X}\rightarrow\Delta$ of curves of genus $g$ with central fibre $\mathscr{X}_0=Y$.  The family is equipped with a line bundle $\mathscr{L}$ of relative degree $d$ and such that $h^0(\mathscr{X}_t,\mathscr{L}_t)>r+1$ for all $t\in\Delta$.  Thus the family $(\pi:\mathscr{X}\rightarrow\Delta,p_i:\Delta\rightarrow\mathscr{X},\mathscr{L})\notin\mathcal{DJ}^{r,d}_{g,n,\mu}(\Delta)$.  Otherwise, if the smooth fibres of $\mathscr{X}\rightarrow\Delta$ admitted de Jonqui\`eres divisors, then by Proposition \ref{prop:closed}, the central fibre would as well.  However, we have just proven this not to be the case, which concludes the induction step.

\begin{rem} $Y$ is a quasi-stable curve obtained via semi-stable reduction from the stable curve $X$ of genus $g$ with no marked points and just one self-intersection node.  Since $X$ is $d$-general (see Definition 4.13 of \autocite{Ca2}), it follows that locally around $X$ the forgetful morphism $\Psi_{d,g,0}:\overline{P}_{d,g,0}\rightarrow\Delta$  (with $\Delta\subset\overline{{\mathcal{M}}}_{g}$) is proper.  Moreover, if $\Psi_{d,g,0}$ is proper, then so are $\Psi_{d,g,n}:\overline{P}_{d,g,n}\rightarrow\Delta$ (with $\Delta\in\overline{\mathcal{M}}_{g,n}$ - see for example the discussion in Sections 7 and 8 of \cite{Me}) and $\mathcal{DJ}^{r,d}_{g,n,\mu}(\Delta)\rightarrow\Delta$.  
\end{rem}

	\section{Expected dimension for de Jonqui\`eres divisors with negative terms}\label{sec:lorentz} 
\indent  It is also worthwhile to study de Jonqui\`eres divisors whose partition $\mu$ of $d$ contains negative terms. In fact, in Section \ref{sec:compactified} we saw that negative coefficients occur naturally when considering  de Jonqui\`eres divisors on nodal stable curves, as the twists $T$ may be negative.  For simplicity of notation, we consider only de Jonqui\`eres divisors with distinct points in the support.

\begin{defi}\label{def:lorentz}
Fix a curve $C$ equipped with a linear series $l\in\Grd$ and let 
\[\mu=(a_1,\ldots,a_{n_1},-b_1,\ldots,-b_{n_2})\] 
be a partition of $d$ of length $n$, where $a_i,b_i$ are positive integers satisfying $\sum_{i=1}^{n_1}a_i - \sum_{i=1}^{n_2}b_i=d$ and $n_1,n_2$ are fixed positive integers with $n_1+n_2=n$.
We define the space $DJ_{n_1,n_2}^{r,d}(\mu,C,l)$ of de Jonquières divisors with $n_1$ positive and $n_2$ negative terms corresponding to the linear series $l$ on the curve $C$ by the rule
\[ \sum_{i=1}^{n_1}a_i p_i - \sum_{i=1}^{n_2}b_i q_i \in DJ_{n_1,n_2}^{r,d}(\mu,C,l)\]
if and only if
\[ \sum_{i=1}^{n_1}a_i p_i \in DJ_{n_1}^{r',d'}(\mu',C,l'), \]
where $p_i,q_i\in C$, $\mu'=(a_1,\ldots,a_{n_1})$ is a positive partition of $d'=\sum_{i=1}^{n_1}a_i=d+\sum_{i=1}^{n_2}b_i$, and $l'$ is the linear series of type $\grdop{r'}{d'}$ given by $l'=l+\sum_{i=1}^{n_2}b_i q_i$.
\end{defi}

\begin{thm}\label{thm:lorentz}
Fix a general curve $C$ of genus $g$ with a general linear series $l=(L,V)\in\Grd$, and let $\mu=(a_1,\ldots,a_{n_1},-b_1,\ldots,-b_{n_2})$ be a partition of $d$ of length $n$, where $a_i,b_i$ are positive integers and $n=n_1+n_2$.  
Assume that the points $q_i$ are general and $l'$ is complete.   If $n_1-d'+r'\geq 0$, then 
 the space $DJ^{r,d}_n(\mu,C,l)$  is of expected dimension $n-d'+r'$.
\end{thm}
\begin{proof}
Set $L'=L(\sum_{i=1}^{n_2}b_i q_i)$.  We first show that $\dim DJ^{r,d}_{n_1}(\mu',C,l')\geq n_1-d'+r'$.  
We distinguish a few cases.  
\begin{itemize}
	\item If $d'=2g-2$ and $L'=K_C$, then $h^0(L') = g$.
	\item If $d'=2g-2$, but $L'\neq K_C$, then $h^0(L') = g-1$.
	\item If $d' > 2g-2$, then $h^0(L') = d'-g+1$.
	\item If $d'<2g-2$, then $h^0(L') = r+\sum_{i=1}^{n_2}b_i+1$, by the generality of the points $q_i$.
\end{itemize}
Note that apart from the case $d'<2g-2$, the assumption that the points $q_i$ are general was not used.
In all cases however, $h^0\bigl(L'|_{\sum_{i=1}^{n_1} a_i p_i}\bigr)=\sum_{i=1}^{n_1}a_i=d'$.
With this in mind, we can describe the space $DJ^{r',d'}_{n_1}(\mu',C,l')$ as the locus in $C_{d'}$ where the vector bundle map $\Phi$ (constructed as in Section \ref{sec:deglocus}, but substituting $L'$ for $L$) has rank at most $h^0(L')-1=r'$.  Hence the lower bound for the dimension of $DJ^{r',d'}_{n_1}(\mu',C,l')$ is given by
\begin{itemize}
	\item $n_1-(h^0(L')-r')(d'-r')=n_1-d'+r'=n_1-g+1$ if $d'= 2g-2$ and $L' = K_C$,
	\item $n_1-(h^0(L')-r')(d'-r')=n_1-g$ if $d'\geq 2g-2$ and $L'\neq K_C$,
	\item $n_1-(h^0(L')-r')(d'-r')=n_1-d+r$ if $d'<2g-2$.
\end{itemize}
The fact that
\[ \dim DJ^{r',d'}_{n_1}(\mu',C,l') = n_1 - d' + r' \]
follows as in the case of effective de Jonqui\`eres divisors, by replacing the occurrences of $L$ by $L'$ in the proof of Theorem \ref{dimension}.  Finally, including the points $q_i$ in the dimension count, we get that the dimension of $DJ_{n_1,n_2}^{r,d}(\mu,C,l)$ is indeed $n-d'+r'$.
\end{proof}

\printbibliography[]

\end{document}